\documentclass{amsart}
\usepackage{amsmath}
\usepackage{amsthm}
\usepackage{amssymb}

\usepackage[utf8]{inputenc}
\usepackage[T1]{fontenc}
\usepackage{lmodern}
\usepackage{eucal}
\pdfoutput=1 
\usepackage{microtype}
\usepackage{url}
\usepackage{enumitem}
\usepackage{tikz-cd}
\usepackage{bbm}
\usepackage{longtable,booktabs}

\usepackage[letterpaper, left = 1.5in, right = 1.5in, top = 1.5in, bottom = 1.5in]{geometry}

\swapnumbers

\usepackage{hyperref}

\usepackage{xcolor}
\hypersetup{
    colorlinks,
    linkcolor={red!50!black},
    citecolor={blue!50!black},
    urlcolor={blue!80!black}
}

\usepackage[capitalise,nameinlink]{cleveref}
\crefname{subsection}{Subsection}{Subsections}
\crefname{subsubsection}{Subsubsection}{Subsubsections}

\theoremstyle{definition}
\newtheorem{theorem}{Theorem}[subsection]

\newtheorem{ex}[theorem]{Example}

\newtheorem{lemma}[theorem]{Lemma}
\newtheorem{prop}[theorem]{Proposition}
\newtheorem{rmk}[theorem]{Remark}

\newtheorem*{rmk*}{Remark}
\newtheorem*{ex*}{Example}
\newtheorem*{theorem*}{Theorem}
\newtheorem*{defn*}{Definition}

\newcommand{\bbZ}{\mathbb{Z}}

\newcommand{\bbF}{\mathbb{F}}

\newcommand{\bbR}{\mathbb{R}}

\newcommand{\bbC}{\mathbb{C}}

\newcommand{\Hom}{\operatorname{Hom}}

\newcommand{\Fib}{\operatorname{Fib}}

\newcommand{\im}{\operatorname{Im}}
\renewcommand{\ker}{\operatorname{Ker}}

\newcommand{\Sq}{\mathrm{Sq}}
\newcommand{\tr}{\operatorname{tr}}
\newcommand{\res}{\operatorname{res}}

\newcommand{\Spin}{\mathrm{Spin}}

\newcommand{\h}{\mathrm{h}}

\newcommand\xqed[1]{%
  \leavevmode\unskip\penalty9999 \hbox{}\nobreak\hfill
  \quad\hbox{#1}}
\newcommand\tqed{\xqed{$\triangleleft$}}

\makeatletter
\DeclareRobustCommand{\tvdots}{%
  \vbox{\baselineskip4\p@\lineskiplimit\z@\kern0\p@\hbox{.}\hbox{.}\hbox{.}}}
\makeatother

\makeatletter
\@namedef{subjclassname@2020}{\textup{2020} Mathematics Subject Classification}
\makeatother

\begin{document}

\title[$K$-theory equivariant with respect to an elementary abelian $2$-group]{$K$-theory equivariant with respect \\ to an elementary abelian $2$-group}

\author[William Balderrama]{William Balderrama}


\begin{abstract}
We compute the $RO(A)$-graded coefficients of $A$-equivariant complex and real topological $K$-theory for $A$ a finite elementary abelian $2$-group, together with all products, transfers, restrictions, power operations, and Adams operations.
\end{abstract}

\maketitle

\section{Introduction}

Fix a finite elementary abelian $2$-group $A$, i.e.\ $A\cong (\bbZ/2)^n$ for some $n$. The purpose of this paper is to provide a reference for the structure of $A$-equivariant complex and real topological $K$-theory. Geometrically, this gives information about stable classes of $A$-equivariant vector bundles over $A$-representation spheres. Homotopically, this gives information about $A$-equivariant stable homotopy theory at chromatic height $1$.

We are not the first to study this. In particular, the additive structure of $\pi_\star KU_A$ is known: Karoubi \cite{karoubi2002equivariant} has described the groups $\pi_\star KU_G$ for any finite group $G$, and the particular case $G=(\bbZ/2)^n$ was revisited by Hu--Kriz \cite{hukriz2006rog}. Moreover, the coefficient ring of the connective spectrum $(ko_{C_2})_2^\wedge$ has been computed by Guillou--Hill--Isaksen--Ravenel \cite{guillouhillisaksenravenel2020cohomology}, and $(KO_{C_2})_2^\wedge$ was studied in \cite{balderrama2021borel}. We are then interested in the descent to $KO_A$ for general $A$ and the wealth of additional structure present, including products, transfers, restrictions, power operations, and Adams operations.

Though our computation gives ostensibly geometric information about vector bundles, our motivation is homotopical. Classically, $KO$ has an infinite Hurewicz image, and its Bott periodicity reflects $v_1$-periodicity in stable homotopy theory. This refines to equivariant Bott periodicity for $\Spin$ representations, which give a rich web of periodicities in $\pi_\star KO_A$, and these suggest a similarly rich web of periodicities in $A$-equivariant stable homotopy theory. Sharper information may be obtained by considering the $A$-equivariant $J$-spectrum $J_A = \Fib\left(\psi^3 - 1 \colon (KO_A)_{(2)}\rightarrow (KO_A)_{(2)}\right)$. Our computation gives information necessary to understand $J_A$, although we shall not pursue this further. When $A = C_2$, we note that the additive structure of $\pi_\star (J_{C_2})_2^\wedge$ may be recovered from work of Adams on the $K$-theory of projective spaces \cite{adams1962vector}, and has been studied in the guise of the $C_2$-equivariant $J$-homomorphism by several people \cite{loffler1977equivariant,crabb1980z2,minami1983equivariant}, with an analysis of the ring structure appearing in \cite{balderrama2021borel}.

We were led to this computation by a different path. Recently, Gepner--Meier \cite{gepnermeier2020equivariant} have produced a fully integral theory of equivariant elliptic cohomology for abelian compact Lie groups, building on work of Lurie \cite{lurie2009survey,lurie2018ellipticii}; this produces good analogues of equivariant $K$-theory at chromatic height $2$. We were initially led to study $A$-equivariant $K$-theory as we were investigating equivariant elliptic cohomology and found that even the height $1$ computations we wished to consult did not exist. From this perspective, $\pi_\star KU_A$ and $\pi_\star KO_A$ give the $A$-equivariant analogues of those height $1$ patterns which are found across chromatic computations at $p=2$. A good understanding of these patterns is necessary for work at higher heights, and this motivated the present work.

We summarize the structure of $\pi_\star KU_A$ in \cref{ssec:kusummary}, and of $\pi_\star KO_A$ in \cref{ssec:kosummary}.

\subsection{Conventions}\label{ssec:preliminaries}

We maintain the following conventions throughout the paper.

(1)~~ We write $A^\vee$ for the dual space of $A$. Thus $A^\vee\cong \Hom(A,C_2)\cong \Hom(A,U(1))$, with corresponding isomorphisms $\bbZ[A^\vee]\cong RO(A) \cong RU(A)$. We write the group structure on $A^\vee$ multiplicatively, and refer to its elements as functionals. Given $K,L\subset A^\vee$, we write $K+L\subset A^\vee$ for the smallest subgroup containing $K\cup L$.

(2)~~ Throughout the paper, the symbols $\lambda,\mu,\kappa,\delta$ are understood to range through linearly independent functionals on $A$. Thus for instance ``$\bbZ[x_{\lambda,\mu}]$'' would be shorthand for ``$\bbZ[x_{\lambda,\mu} : \lambda,\mu\in A^\vee\text{ linearly independent}]$''.

(3)~~ Similarly, the symbol $H$ ranges through the rank $2$ subgroups of $A^\vee$, and the symbol $E$ ranges through the rank $3$ subgroups of $A^\vee$.

(4)~~ Given functionals $\lambda_1,\ldots,\lambda_n\in A^\vee$, we shall write $\langle \lambda_1,\ldots,\lambda_n\rangle\subset A^\vee$ for the subgroup generated by $\lambda_1,\ldots,\lambda_n$.

(5)~~ Finally, we recall the following. Given a codimension $1$ subgroup $j\colon \ker(\lambda)\subset A$, there is an $A$-equivariant cofiber sequence
\begin{center}\begin{tikzcd}
A/\ker(\lambda)_+\ar[r]&S^0\ar[r,"\rho_\lambda"]&S^\lambda
\end{tikzcd},\end{center}
where the first map sends $A/\ker(\lambda)$ to the non-basepoint of $S^0$ and the second map is the inclusion of poles. Smashing this with an $A$-equivariant spectrum $R_A$ gives rise to the long exact sequence
\begin{center}\begin{tikzcd}
\cdots\ar[r]&\pi_{j^\ast(\star)}R_{\ker(\lambda)}\ar[r,"j_!"]&\pi_\star R_A\ar[r,"\rho_\lambda"]&\pi_{\star-\lambda}R_A\ar[r,"j^\ast"]&\pi_{j^\ast(\star-\lambda)}R_{\ker(\lambda)}\ar[r]&\cdots\end{tikzcd}\end{center}
with $j^\ast$ the restriction and $j_!$ the transfer, and we shall freely make use of the resulting equalities
\begin{gather*}
\ker(j^\ast\colon \pi_\star R_A\rightarrow\pi_{j^\ast\star}R_{\ker(\lambda)}) = \im(\rho_\lambda\colon \pi_{\star+\lambda}R_A\rightarrow\pi_\star R_A), \\
\ker(\rho_\lambda\colon \pi_{\star+\lambda}R_A\rightarrow\pi_\star R_A) = \im(j_!\colon \pi_{j^\ast(\star+\lambda)}R_{\ker(\lambda)}\rightarrow \pi_{\star+\lambda} R_A).
\end{gather*}

\section{Complex \texorpdfstring{$K$}{K}-theory}

\subsection{Summary} \label{ssec:kusummary}

For ease of reference, we gather the result of our computation in one place.

\begin{theorem}\label{thm:kua}
The coefficients of $KU_A$ behave as described in this subsection.
\tqed
\end{theorem}

The proof is spread throughout the rest of this section, glued together as described below. 

\subsubsection{Generators}\label{sssec:generators}

We begin by describing a set of multiplicative generators for $\pi_\star KU_A$. 

There are three basic types of invertible elements in $\pi_\star KU_A$ arising from equivariant Bott periodicity. Following Atiyah \cite{atiyah1968bott}, for every orthogonal $A$-representation $V$ admitting a $\Spin^c$ structure, there is an invertible Bott class $b_V\in \pi_V KU_A$; to be precise, we shall take the Bott class denoted there by $\lambda_V$. In particular, let $\beta = b_2 \in \pi_2 KU_A$ be the standard Bott class, and define the following Thom classes. First, for every nontrivial functional $\lambda\in A^\vee$, the orthogonal representation $2\lambda = \lambda\otimes\bbC$ admits a complex structure, and we set $\tau_\lambda^2 = \beta \cdot b_{2\lambda}^{-1}$. Next, for every rank $3$ subgroup $E\subset A^\vee$, the orthogonal representation $\sum_{\lambda\in E}\lambda$ admits a $\Spin$ structure, and we set $\tau_E = \beta^4\cdot b_{\Sigma_{\lambda\in E}\lambda}^{-1}$. Let us agree to call any class in $\pi_\star KU_A$ which is a product of classes of the form $\beta^{\pm 1}$, $\tau_\lambda^{\pm 2}$, and $\tau_E$, a \textit{Bott class}.

There are two basic types of noninvertible elements in $\pi_\star KU_A$. First are classes obtained from the case where $A$ is cyclic: for each nontrivial functional $\lambda\in A^\vee$, there is a class $\rho_\lambda\in \pi_{-\lambda} KU_A$ obtained as the Hurewicz image of the class in $\pi_{-\lambda}S_A$ represented by the inclusion of poles $S^0\rightarrow S^\lambda$. Second are classes present only when $A$ is of rank at least $2$: for each rank $2$ subgroup $H\subset A^\vee$, there is a unique class $k_H\in \pi_{4-\Sigma_{\lambda\in H}\lambda}KU_A$ such that $2k_H = \tr(1)$, where $\tr\colon \pi_0 KU\rightarrow \pi_{4-\Sigma_{\lambda\in H}\lambda}KU_H$ is the transfer. We will construct $k_H$ in \cref{lem:halftransfer}.

We also give names to the following elements of $\pi_0 KU_A$:
\[
d_\lambda = \rho_\lambda^2\tau_\lambda^{-2}\beta,\qquad \sigma_\lambda = 1 - d_\lambda,\qquad h_\lambda = 1 + \sigma_\lambda.
\]
Under the isomorphism $\pi_0 KU_A\cong RU(A)$, the class $\sigma_\lambda$ corresponds to the character $\lambda\otimes\bbC$, and $h_\lambda = \bbC[A/\ker(\lambda)]$.

\subsubsection{Basis}\label{sssec:basis}

If $\pi_\xi KU_A\neq 0$, then there is a nonzero class $x\in \pi_\xi KU_A$ of the form
\[
x = \rho_{\lambda_1}\cdots \rho_{\lambda_n} \cdot t \cdot k_{H_1}\cdots k_{H_m}
\]
satisfying the following conditions. Write $K = H_1+\cdots+H_m$ and $L = \langle \lambda_1,\ldots,\lambda_n\rangle$. Then
\begin{enumerate}
\item $\lambda_1,\ldots,\lambda_n\in A^\vee$ are linearly independent;
\item $t$ is a Bott class;
\item $K$ is of rank $2m$;
\item $K\cap L = 0$.
\end{enumerate}
We shall call such a monomial a \textit{basic monomial}, and refer to the classes represented by basic monomials as \textit{basic generators}. These representations are not unique. We now have
\begin{align*}
\pi_\xi KU_A = \bbZ\{x\}\otimes RU(A)/\left((\sigma_{\lambda}+1):\lambda\in \{\lambda_1,\ldots,\lambda_n\},~(\sigma_\lambda-1):\lambda\in K\right)
\end{align*}
as $\pi_0 KU_A$-modules. This is largely a reinterpretation of the computation of $\pi_\star KU_A$ by Hu--Kriz \cite{hukriz2006rog}, as we will explain in \cref{ssec:basis}.

\begin{rmk}
By relation \ref{r9} below, one may always suppose a basic generator is represented by a basic monomial as above satisfying $n\leq 2$. Alternately, if $n\neq 0$, then one may suppose $m = 0$.
\tqed
\end{rmk}

\begin{ex}
If $\xi = 3 - \lambda - \mu - \lambda \mu - \kappa - \lambda\mu\kappa$, then $x = \rho_{\lambda\kappa}\rho_{\mu\kappa}\tau_{\langle\lambda,\mu,\kappa\rangle}\tau_{\lambda\kappa}^{-2}\tau_{\mu\kappa}^{-2}$.
\tqed
\end{ex}

\subsubsection{Relations}\label{sssec:relations}

The multiplicative structure of $\pi_\star KU_A$ is determined by the following:

\begin{enumerate}[label={R.\arabic*}]
\item \label{r1} All basic monomials (\ref{sssec:basis}) in the same degree represent the same class;
\item \label{r2} $\rho_\lambda h_\lambda = 0$, or equivalently, $\sigma_\lambda\rho_\lambda = - \rho_\lambda$;
\item \label{r3} $d_{\lambda\mu} = d_\lambda+d_\mu-d_\lambda d_\mu$, or equivalently, $\sigma_{\lambda\mu} = \sigma_\lambda\sigma_\mu$;
\item \label{r4} $\rho_\lambda\rho_\mu\rho_{\lambda\mu}=0$;
\item \label{r5} $\rho_\lambda k_H = 0$ for $\lambda\in H$;
\item \label{r6} $k_{\langle \lambda,\mu\rangle} k_{\langle \lambda,\kappa\rangle} = 2 \tau_{\langle\lambda,\mu,\kappa\rangle}\tau_{\mu\kappa}^{-2}\tau_{\lambda\mu\kappa}^{-2}k_{\langle \lambda,\mu\kappa\rangle} - \rho_\mu\rho_\kappa\rho_{\lambda\mu}\rho_{\lambda\kappa}\tau_\lambda^2\beta^2$;
\item \label{r7} $k_{\langle\lambda,\mu\rangle}^2 = \tau_\lambda^2\tau_\mu^2\tau_{\lambda\mu}^2 h_\lambda h_\mu$.
\end{enumerate}
This will be shown in \cref{ssec:relations}.

\begin{ex}\label{ex:basic}
We record the following special cases of \ref{r1}:
\begin{enumerate}[label={R.\arabic*},resume]
\item \label{r8} $\tau_{\langle\lambda,\mu,\kappa\rangle}^2 = \tau_\lambda^2\tau_\mu^2\tau_\kappa^2\tau_{\lambda\mu}^2\tau_{\lambda\kappa}^2\tau_{\mu\kappa}^2\tau_{\lambda\mu\kappa}^2$;
\item \label{r9} $\rho_\lambda\rho_\mu\rho_\kappa\tau_{\langle\lambda,\mu,\kappa\rangle}\beta = \rho_{\lambda\mu\kappa}\tau_\lambda^2\tau_\mu^2\tau_\kappa^2 k_{\{1,\lambda\mu,\lambda\kappa,\mu\kappa\}}$;
\item \label{r10} $\tau_{\langle\mu,\kappa,\delta\rangle}\tau_{\lambda\kappa}^2\tau_{\lambda\mu\kappa}^2 k_{\langle\lambda,\mu\rangle} k_{\langle\kappa,\delta\rangle} = \tau_{\langle\lambda,\mu,\kappa\rangle}\tau_\delta^2\tau_{\delta\kappa}^2 k_{\langle\kappa,\mu\delta\rangle} k_{\langle\lambda\kappa,\mu\rangle}$;
\item \label{r11} $\rho_\lambda\rho_\mu \tau_{\langle\kappa,\lambda\mu,\delta\rangle}\tau_{\lambda\delta}^2\tau_{\mu\delta}^2 k_{\langle\lambda\mu\kappa,\delta\rangle} = \rho_{\lambda\delta}\rho_{\mu\delta} \tau_{\langle\lambda,\mu,\delta\rangle}\tau_{\lambda\mu\kappa}^2\tau_{\lambda\mu\kappa\delta}^2k_{\langle\kappa,\delta\rangle}$.
\end{enumerate}
Here, \ref{r11} is redundant, being implied by \ref{r9}. It is plausible that \ref{r1} could be replaced by some minimal set of relations such as these, but we shall not pursue this.
\tqed
\end{ex}

\begin{rmk}
\ref{r6} may rewritten as $k_{\langle\lambda,\mu\rangle}k_{\langle\lambda,\kappa\rangle} = h_\mu\cdot \tau_{\langle\lambda,\mu,\kappa\rangle}\tau_{\mu\kappa}^{-2}\tau_{\lambda\mu\kappa}^{-2}k_{\langle\lambda,\mu\kappa\rangle}$, although this is no longer symmetric.
\tqed
\end{rmk}

\begin{rmk}
It is interesting to observe that relations \ref{r2} and \ref{r3} do not imply \ref{r4}, but do imply $2\rho_\lambda\rho_\mu\rho_{\lambda\mu} = 0$, and that this is all that holds in $\pi_\star KO_A$ (\cref{sec:ko}).
\tqed
\end{rmk}

\subsubsection{Restrictions}

Fix a second elementary abelian $2$-group $B$. For any homomorphism $g\colon A\rightarrow B$ there is a restriction
\[
g^\ast\colon \pi_\star KU_B\rightarrow \pi_{g^{\ast}\star}KU_A.
\]
This is determined by the following.
\begin{enumerate}
\item $g^\ast$ is a ring homomorphism;
\item $g^\ast$ preserves Bott classes;
\item $g^\ast(\rho_\lambda) = \rho_{g^\ast\lambda}$, with the understanding that $\rho_1=0$;
\item If $g^\vee H\subset A^\vee$ is of rank $2$, then $g^\ast(k_H) = k_{g^\vee H}$; if $g^\vee H\subset A^\vee$ is cyclic with generator $\lambda$, then $g^\ast(k_H) = \tau^2_\lambda h_\lambda$; and if $g^\vee H \subset A^\vee$ is trivial then $g^\ast(k_H) = 2$.
\end{enumerate}
Here, (4) holds by \cref{lem:halftransfer} and the definition of $k_H$, and the rest are clear.

\begin{ex}
Write $j\colon \ker(\lambda)\subset A$. Then $j^\ast(\tau_{\langle\lambda,\mu,\kappa\rangle} )= \tau_{j^\ast(\mu)}^2\tau_{j^\ast(\kappa)}^2\tau_{j^\ast(\mu\kappa)}^2$.
\tqed
\end{ex}

\subsubsection{Transfers}\label{sssec:transfers}

To any subgroup inclusion $j\colon L\subset A$, there is a transfer $j_!\colon \pi_{j^\ast \star} KU_L\rightarrow\pi_\star KU_A$. These are transitive, so to describe their effect it is sufficient to consider the case where $L = \ker(\lambda)$ is a codimension $1$ subgroup. Now $j_!$ is determined by the following.

\begin{enumerate}[label={T.\arabic*}]
\item \label{t1} $j_!$ is $\pi_\star KU_A$-linear, i.e.\ $j_!(x\cdot j^\ast(y)) = j_!(x)\cdot y$ for $x\in \pi_{j^\ast \star}KU_{\ker(\lambda)}$ and $y\in \pi_\star KU_A$;
\item \label{t2} $j_!\colon \pi_0 KU_{\ker(\lambda)}\rightarrow \pi_0 KU_A$ satisfies $j_!(1) = h_\lambda\in \pi_0 KU_A$;
\item \label{t3} $j_!\colon \pi_{2-2 j^\ast\mu}KU_{\ker(\lambda)}\rightarrow\pi_{2-\mu-\lambda\mu}KU_A$ satisfies $j_!(\tau_{j^\ast\mu}^2) = \rho_\mu\rho_{\lambda\mu}\beta$;
\item \label{t4} $j_!\colon \pi_{2-2j^\ast\mu}KU_{\ker(\lambda)}\rightarrow\pi_{3-\lambda-\mu-\lambda\mu}KU_A$ satisfies $j_!(\tau_{j^\ast\mu}^2) = k_{\langle\lambda,\mu\rangle}$.
\end{enumerate}
This will be shown in \cref{ssec:transfer}.

\begin{ex}\label{ex:specialtransfer}
\hphantom{blank}
\begin{enumerate}[label={T.\arabic*},resume]
\item \label{t5} $j_!\colon \pi_0 KU_{\ker(\lambda)}\rightarrow \pi_{\mu-\lambda\mu+\kappa-\lambda\kappa}KU_A$ satisfies $j_!(1) =\tau_{\langle\lambda,\mu,\kappa\rangle}^{-1}\tau_{\lambda\mu}^2\tau_{\lambda\kappa}^2 k_{\langle\lambda,\mu\kappa\rangle}$.
\end{enumerate}
This follows from \ref{t1} and \ref{t4}, using $1=j^\ast(\tau_{\langle\lambda,\mu,\kappa\rangle}^{-1}\tau_{\lambda\mu}^2\tau_{\lambda\kappa}^2\tau_{\mu\kappa}^2)$.
\tqed
\end{ex}

\subsubsection{Weyl action}\label{sssec:weyl}

For any subgroup $j\colon L\subset A$, there is an action of the Weyl group $W_AL=A/L$ on $\pi_{j^\ast\star}KU_L$. Together with all the preceding, this makes the collection $\{\pi_\star KU_L:L\subset A\}$ into an $RO(A)$-graded Green functor \cite{green1971axiomatic,lewismandell2006equivariant}. To describe this action we may reduce to the case where $L = \ker(\lambda)$ is a codimension $1$ subgroup, so that $W_A L$ is cyclic with generator $Q$. Now $Q$ acts by
\[
Q x = j^\ast j_!(x) - x.
\]
This is merely a reformulation of the double coset formula.

\begin{ex}
If $\xi = i^\ast(\zeta)$ for any section $i\colon A\rightarrow L$ and $\zeta\in RO(L)$, then $Q$ acts trivially on $\pi_{j^\ast\xi} KU_L = \pi_\zeta KU_L$. On the other hand,
\begin{enumerate}
\item $Q$ acts on $\pi_{j^\ast(2-\mu-\lambda\mu)} KU_L= \bbZ\{\tau_{j^\ast\mu}^2\}\otimes RU(L)$ as multiplication by $-\sigma_{j^\ast\mu}$ (\ref{t3});
\item $Q$ acts on $\pi_{j^\ast(3-\lambda-\mu-\lambda\mu)}KU_L=\bbZ\{\tau_{j^\ast\mu}^2\}\otimes RU(L)$ as multiplication by $\sigma_{j^\ast\mu}$ (\ref{t4}).
\tqed
\end{enumerate}
\end{ex}

\subsubsection{Power operations}

Equivariant $K$-theory is equipped with power operations, as constructed by Atiyah \cite{atiyah1966power}. From this, one may produce for every subgroup $j\colon L\subset A$ a multiplicative norm map 
\[
j_\otimes\colon \pi_\star KU_L\rightarrow \pi_{j_!\star}KU_A.
\]
Together with all the preceding, these norms make $\{\pi_\star KU_L:L\subset A\}$ into some flavor of Tambara functor \cite{tambara1993multiplicative,angeltveitbohmann2018graded}, although we shall not make use of this formalism. By transitivity, to describe this it is sufficient to instead describe the external squaring operation
\[
\Sq\colon \pi_\star KU_A\rightarrow \pi_{\star(1+\sigma)} KU_{A\times C_2},
\]
where $\sigma$ denotes the generating functional on $C_2$. This is determined by the following.
\begin{enumerate}
\item $\Sq(xy) = \Sq(x)\Sq(y)$;
\item $\Sq(x+y) = \Sq(x)+\Sq(y) + \tr(xy)$, where $\tr$ is the transfer;
\item $\Sq$ preserves Bott classes;
\item $\Sq(\rho_\lambda)=\rho_\lambda\rho_{\lambda\sigma}$;
\item $\Sq(k_{\langle\lambda,\mu\rangle}) = \tau_{\langle\lambda,\mu,\sigma\rangle}\tau_\sigma^{-4}(\sigma_\lambda+\sigma_\mu+\sigma_{\lambda\mu}+\sigma_\sigma)$.
\end{enumerate}
Here, (1) and (2) are general properties of $\Sq$, and the rest will be computed in \cref{ssec:squaring}.

\begin{rmk}
Regarding (3), explicitly we have
\begin{gather*}
\Sq(\beta) = \tau_\sigma^{-2}\beta^2,\qquad \Sq(\tau_\lambda^2) = \tau_\lambda^2\tau_{\lambda\sigma}^2\tau_\sigma^{-2}, \\
\Sq(\tau_{\langle\lambda,\mu,\kappa\rangle}) = \tau_{\langle\lambda,\mu,\kappa\sigma\rangle}\tau_{\langle\lambda,\mu,\kappa\rangle}\tau_{\langle\lambda,\mu,\sigma\rangle}\tau_\lambda^{-2}\tau_{\mu}^{-2}\tau_{\lambda\mu}^{-2}\tau_\sigma^{-8},
\end{gather*}
where the last monomial is noncanonical though the class itself is not.
\tqed
\end{rmk}

\subsubsection{Adams operations}

Fix an odd integer $\ell$. Then the Adams operation $\psi^\ell$ acts on $\pi_\star KU_A[\tfrac{1}{\ell}]$ by ring automorphisms \cite{hiratakono1982bott}, and is given on generators by the following.

\begin{enumerate}
\item 
$
\psi^\ell(\beta) = \ell \beta.
$
\item 
$
\psi^\ell(\tau_\lambda^2) = \tau_\lambda^2(1+\tfrac{1}{2}(\ell-1)d_\lambda).
$
\item 
$
\psi^\ell(\tau_E) = \tau_E(1+\tfrac{1}{8}(\ell^2-1)\sum_{\lambda\in E\setminus \{1\}}d_\lambda).
$
\item 
$
\psi^\ell(\rho_\lambda)=\rho_\lambda.
$
\item 
$
\psi^\ell(k_H) = k_H.
$
\end{enumerate}

Here, (1) is standard, and (4) is clear; (2) is pulled back from \cref{lem:cyclicadams} along $\lambda\colon A\rightarrow C_2$, and we shall see (3) in \cref{lem:psie} and (5) in \cref{lem:halftransfer}.

\begin{ex}
$\psi^{-1}(\tau_\lambda^2)=\tau_\lambda^2\sigma_\lambda$ and $\psi^{-1}(\tau_E) = \tau_E$.
\tqed
\end{ex}

\begin{rmk}
Using \ref{r3}, one may rewrite $\psi^\ell(\tau_{\langle\lambda,\mu,\kappa\rangle})$ as
\[
\tau_{\langle\lambda,\mu,\kappa\rangle}(1+\tfrac{1}{2}(\ell^2-1)(d_\lambda+d_\mu+d_\kappa)-\tfrac{1}{4}(\ell^2-1)(d_\lambda d_\mu+d_\lambda d_\kappa+d_\mu d_\kappa)+\tfrac{1}{8}(\ell^2-1)d_\lambda d_\mu d_\kappa).
\]
\tqed
\end{rmk}

This concludes our statement of \cref{thm:kua}.

\subsection{Low ranks}\label{ssec:lowranks}

Let $\sigma$ be the generating functional of $C_2$. We begin by considering $KU_{C_2}$; here we omit the subscript $\sigma$ from the classes in $\pi_\star KU_{C_2}$ introduced in \cref{sssec:generators}. For this material see also \cite{balderrama2021borel}.

\begin{lemma}\label{lem:cyclic}
$
\pi_\star KU_{C_2} = \bbZ[\beta^{\pm 1},\tau^{\pm 2},\rho]/(\rho\cdot h).
$
\end{lemma}
\begin{proof}
As in \cref{ssec:preliminaries}, there is a $C_2$-equivariant cofiber sequence
\[
C_{2+}\rightarrow S^0\rightarrow S^\sigma
\]
giving rise to a long exact sequence
\begin{center}\begin{tikzcd}
\cdots\ar[r]&\pi_{\star+\sigma} KU_{C_2}\ar[r,"\rho"]&\pi_\star KU_{C_2}\ar[r,"\res"]&\pi_{\dim \star} KU\ar[r,"\tr"]&\pi_{\star+\sigma-1} KU_{C_2}\ar[r]&\cdots
\end{tikzcd}.\end{center}
In particular, there is a short exact sequence
\begin{center}\begin{tikzcd}
0\ar[r]&\pi_0 KU\ar[r,"\tr"]&\pi_0 KU_{C_2}\ar[r,"\rho"]&\pi_{-\sigma}KU_{C_2}\ar[r]&0
\end{tikzcd}.\end{center}
As $\tr(1) = \bbC[C_2] = h$, we have $\rho\cdot h = 0$. This sequence also implies $\pi_{-\sigma}KU_{C_2} = \bbZ\{\rho\}$, and the lemma follows.
\end{proof}

\begin{lemma}\label{lem:cyclicadams}
The Adams operation $\psi^\ell$ for $\ell$ odd acts on $\pi_\star KU_{C_2}[\tfrac{1}{\ell}]$ by multiplicative automorphisms, and is given on generators by
\[
\psi^\ell(\beta) = \ell\beta,\qquad \psi^\ell(\tau^{-2}) = \tau^{-2}(1+\tfrac{1}{2}(\ell^{-1}-1)d),\qquad \psi^{\ell}(\rho) = \rho.
\]
\end{lemma}
\begin{proof}
As $\pi_\star KU_{C_2}[\tfrac{1}{\ell}]$ embeds into $\pi_\star (KU_{C_2})_2^\wedge$, it suffices to show these identities hold after $2$-completion. In \cite[Theorem 7.3]{adams1962vector}, Adams computes the $K$-theory of stunted real projective spaces, together with their action by $\psi^\ell$. This computation implies that the completion map $\pi_{\ast+2\sigma}KU_{C_2}\rightarrow KU^{-\ast}(P^\infty_2)\cong \lim_{n\rightarrow\infty}KU^{-\ast}(P^n_2)$, where $P^n_2 = \bbR P^n/\bbR P^1$, is an isomorphism after $2$-completion. Thus Adams' computation gives us the action of $\psi^\ell$ on $\pi_\star (KU_{C_2})_2^\wedge$, after noting that his $\bar{\nu}^{(1)}$ corresponds to our $-\tau^{-2}\beta$ and his $\nu^{(2)}$ corresponds to our $\tau^{-2}\beta d$.
\end{proof}

We also record the following here.

\begin{lemma}\label{lem:psie}
Suppose that $A$ is of arbitrary rank, and fix an odd integer $\ell$. Then the action of $\psi^\ell$ on $\pi_\star KU_A[\tfrac{1}{\ell}]$ satisfies
\[
\psi^\ell(\tau_E) = \tau_E\left(1+\tfrac{1}{8}(\ell^2-1)\sum_{\lambda\in E\setminus \{1\}}d_\lambda\right).
\]
\end{lemma}
\begin{proof}
Write $\xi = 8 - \sum_{\lambda\in E}\lambda$. The joint restriction map
\[
\pi_\xi KU_A\rightarrow \prod_{\substack{i\colon L\subset A ,\\ L\text{ cyclic}}} \pi_{i^\ast \xi}KU_L
\]
is an injection, so it is sufficient to verify the stated formula for $\psi^\ell$ after restriction to any cyclic subgroup of $A$. This now follows from \cref{lem:cyclicadams}.
\end{proof}

Now suppose that $A$ is of rank $2$.

\begin{lemma}\label{lem:halftransfer}
Write $\xi = 4 - \sum_{\lambda\in A^\vee}\lambda$. Then
\[
\pi_{\xi} KU_A = \bbZ\{k\},\qquad \pi_{\xi+1}KU_A = 0,
\]
where $k$ satisfies the following properties. Choose any $j\colon\ker(\lambda)\subset A^\vee$. Identify $\ker(\lambda)\cong C_2$, and write $i\colon 1\subset C_2$. Note $j^\ast(\xi) = 2-2\sigma$, where $\sigma$ is the generating functional on $C_2$.

\begin{enumerate}
\item $k = j_!(\tau^2)$, where $j_!\colon \pi_{2-2\sigma}KU_{C_2}\rightarrow\pi_{\xi}KU_A$;
\item $2k = \tr(1)$, where $\tr\colon \pi_0 KU\rightarrow \pi_{\xi} KU_A$ is the transfer;
\item $k$ restricts to $2$ in $\pi_0 KU$;
\item $\psi^\ell(k) = k$ in $\pi_\star KU_A[\tfrac{1}{\ell}]$;
\item $k$ restricts to $\tau^2 h$ in $\pi_{2-2\sigma}KU_{C_2}$.
\end{enumerate}
\end{lemma}
\begin{proof}
(1)~~ Choose $\mu\in A^\vee$ linearly independent from $\lambda$, so that $A^\vee = \langle\lambda,\mu\rangle$. The cofibering
\[
A/\ker(\lambda)_+\otimes S^{\xi-\lambda}\rightarrow S^{\xi-\lambda}\rightarrow S^{\xi},
\]
gives a short exact sequence
\begin{center}\begin{tikzcd}
0\ar[r]&\pi_{\xi+(1-\lambda)}KU_A\ar[r,"j^\ast"]&\pi_{2-2\sigma} KU_{C_2}\ar[r,"j_!"]&\pi_{\xi}KU_A\ar[r]&0
\end{tikzcd}.\end{center}
As $\xi + (1-\lambda) = -\mu-\lambda\mu+(4-2\lambda)$, we may identify
\[
\pi_{\xi+(1-\lambda)}KU_A = \pi_{-\mu-\lambda\mu}KU_A\otimes\bbZ\{\tau_\lambda^2\beta\} = \bbZ\{\rho_\mu\rho_{\lambda\mu}\tau_\lambda^2\beta\}.
\]
As
\[
\pi_{2-2\sigma}KU_{C_2}=\bbZ\{\tau^2,\rho^2\beta\},\qquad j^\ast(\rho_\mu\rho_{\lambda\mu}\tau_\lambda^2\beta) = \rho^2\beta,
\]
it follows that $\pi_{\xi} KU_A = \bbZ\{k\}$ where $k=j_!(\tau^2)$. The same cofibering shows also $\pi_{\xi+1}KU_A = 0$.

(2)~~Note that $i_!\colon \pi_0 KU\rightarrow\pi_{2-2\sigma}KU_{C_2}$ satisfies $i_!(1) = i_!(i^\ast(\tau^2)) = \tau^2\cdot i_!(i^\ast(1)) = \tau^2 h$. By transitivity and the short exact sequence used for (1), it follows that $\tr\colon \pi_0 KU\rightarrow \pi_\xi KU_A$ satisfies
\[
\tr(1) = j_! i_!(1) = j_!(\tau^2 h) = j_!(2\tau^2) = 2k.
\]

(3)~~This follows from the double coset formula, as $A$ acts trivially on $\pi_{(ji)^\ast(\xi)}KU$.

(4)~~$2k$ is in the Hurewicz image by (2), so is fixed by $\psi^\ell$. Thus the same is true for $k$.

(5)~~As $k$ is fixed by $\psi^{-1}$, its restriction to $\pi_{2-2\sigma} KU_{C_2}$ lands in the fixed submodule $H^0(\{\psi^{\pm 1}\};\pi_{2-2\sigma}KU_A) = \bbZ\{\tau^2 h\}$. Thus $j^\ast(k) = \ell\cdot \tau^2 h$ for some integer $\ell$, and $\ell = 1$ by (3).
\end{proof}

For more general $A$, we obtain the class $k_H\in \pi_{4-\Sigma_{\lambda\in H}\lambda} KU_A$ by restriction along $A\rightarrow H^\vee$.

\subsection{Basis}\label{ssec:basis}

Now let $A$ be an arbitrary finite elementary abelian $2$-group. The structure of $\pi_\star KU_A$ was investigated by Hu--Kriz in \cite{hukriz2006rog}; part of their argument can be understood as a constructive proof of the following.

\begin{lemma}\label{lem:hukriz}
Every $\xi \in RO(A)$ may be written in the form
$
\xi = \epsilon + S + V,
$
where
\begin{enumerate}
\item $\epsilon\in\{0,1\}$.
\item $S$ is a sum of virtual representations of the form $\pm 2\lambda$ and $\pm \sum_{\lambda\in E}\lambda$. In particular, $S$ is $KU$-orientable.
\item $V$ is of the form $V = \sum_{1\leq i \leq n}\lambda_i + \sum_{1\leq j \leq m}\sum_{\lambda\in H_j}\lambda$, where
\begin{enumerate}
\item $\lambda_1,\ldots,\lambda_n\in A^\vee$ are linearly independent;
\item $H_1,\ldots,H_m\subset A^\vee$ are of rank $2$ and $H_1+\cdots+H_m\subset A^\vee$ is of rank $2m$;
\item $\langle\lambda_1,\ldots,\lambda_n\rangle\cap (H_1+\cdots+H_m)=0$.
\end{enumerate}
\end{enumerate}
\end{lemma}
\begin{proof}
This is contained within the proof of \cite[Theorem 1]{hukriz2006rog}, so let us just explain how to translate their work to the present context. Let $I\subset RO(A)$ be the subgroup generated by the trivial representation together with elements of the form $S$ given in (2). Then we are claiming that every element of $RO(A)/I$ is equivalent to one of the form $V$ given in (3).

Additively, we may identify $RO(A)/\bbZ\{1\} \cong \bbZ[A^\vee\setminus\{1\}]$, and there is a sequence of surjections
\[
RO(A)\rightarrow\bbF_2[A^\vee\setminus\{1\}]\rightarrow RO(A)/I.
\]
Choose a basis $A^\vee\cong\bbF_2\{\alpha_1,\ldots,\alpha_p\}$. Then $\bbF_2[A^\vee\setminus\{1\}]$ corresponds to the set of hypergraphs on $\{\alpha_1,\ldots,\alpha_p\}$ used in \cite{hukriz2006rog}. There it is shown that every hypergraph on $\{\alpha_1,\ldots,\alpha_p\}$ is equivalent in $RO(A)/I$ to a disjoint union of hypergraphs on subsets of $\{\alpha_1,\ldots,\alpha_p\}$ of cardinality at most $2$. The only hypergraphs on a set $\{\alpha,\beta\}$ with two elements are, in additive notation,
\[
0,\quad \alpha,\quad \beta,\quad \alpha\beta,\quad \alpha+\beta,\quad\alpha+\alpha\beta,\quad\beta+\alpha\beta,\quad\alpha+\beta+\alpha\beta,
\]
and a disjoint union of these is, up to adding a multiple of the trivial representation, of the form $V$ given in (3).
\end{proof}

Recall that a \textit{basic monomial} is a monomial of the form $\rho_{\lambda_1}\cdots\rho_{\lambda_n} \cdot t \cdot k_{H_1}\cdots k_{H_m}$ where $\lambda_1,\ldots,\lambda_n$ are linearly independent, $t$ is a Bott class, $H_1+\cdots+H_m$ is of rank $2m$, and $\langle\lambda_1,\ldots,\lambda_n\rangle\cap (H_1+\cdots+H_m) = 0$, and that a \textit{basic generator} is a class which may be represented by a basic monomial.

\begin{prop}\label{lem:basis}
Fix $\xi\in RO(A)$, and suppose that $\pi_\xi KU_A \neq 0$.
\begin{enumerate}
\item $\pi_{\xi+1}KU_A = 0$;
\item $\pi_\xi KU_A$ is a cyclic $RU(A)$-module generated by a basic generator $x$;
\item Choose a presentation $x = \rho_{\lambda_1}\cdots\rho_{\lambda_n}\cdot t \cdot k_{H_1}\cdots k_{H_m}$ of $x$ by a basic monomial. Then $\pi_\xi KU_A = \bbZ\{x\}\otimes RU(A)/(\sigma_{\lambda}+1:\lambda\in\{\lambda_1,\ldots,\lambda_n\},~\sigma_{\lambda}-1:\lambda\in H_1+\cdots+H_m)$.
\end{enumerate}
\end{prop}
\begin{proof}
First consider a representation $V = \sum_{1\leq i \leq n} (-\lambda_i) + \sum_{1\leq j \leq m}(4-\sum_{\lambda\in H_j}\lambda)$ satisfying the conditions necessary for $y = \rho_{\lambda_1}\cdots\rho_{\lambda_n}k_{H_1}\cdots k_{H_m}\in \pi_V KU_A$ to be a basic monomial. Write $L_i$ and $K_j$ for the quotients of $A$ dual to $\langle \lambda_i\rangle$ and $H_j$, so that the low rank calculations of \cref{ssec:lowranks} imply 
\[
\pi_{-\lambda_i}KU_{L_i}= \bbZ\{\rho_{\lambda_i}\},\qquad \sigma_{\lambda_i}\rho_{\lambda_i} = -\rho_{\lambda_i}
\]
and 
\[
\pi_{4-\Sigma_{\lambda\in H_j}\lambda}KU_{K_j}=\bbZ\{k_{H_j}\},\qquad \rho_{\lambda} k_{H_j} = k_{H_j},
\]
the latter for $\lambda\in H_j$. Let $C^\vee = \langle \lambda_1,\ldots,\lambda_n\rangle+H_1+\cdots+H_m\subset A^\vee$, and choose a splitting of the surjection $A\rightarrow C$ with complementary summand $B$. Then the external K\"unneth ismorphisms of the form $\pi_{\star'} KU_{A'}\otimes_{\pi_\ast KU}\pi_{\star''} KU_{A''}\cong \pi_{\star'+\star''}KU_{A'\oplus A''}$ imply that
\begin{align*}
\pi_V KU_A &= \pi_0 KU_B \otimes \pi_V KU_C = RU(B)\otimes\bbZ\{y\} \\
&= \bbZ\{y\}\otimes RU(A)/\left(\sigma_\lambda+1:\lambda\in\{\lambda_1,\ldots,\lambda_n\},\,\sigma_\lambda-1:\lambda\in H_1+\cdots+H_m\right)
\end{align*}
as $\pi_0 KU_A$-modules, and that $\pi_{V+1} KU_A = 0$. This proves the lemma when $\xi = V$. The general case then follows from \cref{lem:hukriz}, which implies that any $\xi \in RO(A)$ may be written in the form $\epsilon + S + V$ where $\epsilon\in\{0,1\}$ and $\pi_S KU_A$ contains a Bott class $t$.
\end{proof}

We must verify the uniqueness of basic generators.

\begin{lemma}[\ref{r4}]\label{lem:r4}
$\rho_\lambda\rho_\mu\rho_{\lambda\mu}=0$.
\end{lemma}
\begin{proof}
Note that $\beta^{-2}k_{\langle\lambda,\mu\rangle}\in\pi_{-1-\lambda-\mu-\lambda\mu}KU_A$. In particular $\pi_{-1-\lambda-\mu-\lambda\mu}KU_A\neq 0$, and thus $\pi_{-\lambda-\mu-\lambda\mu}KU_A = 0$ by \cref{lem:basis}. This implies $\rho_\lambda\rho_\mu\rho_{\lambda\mu}=0$.
\end{proof}

\begin{lemma}[\ref{r9}]\label{lem:r9}
$
\rho_\lambda\rho_\mu\rho_\kappa = \rho_{\lambda\mu\kappa}\tau_{\langle\lambda,\mu,\kappa\rangle}^{-1}\tau_\lambda^2\tau_\mu^2\tau_\kappa^2 \beta^{-1} k_{\{1,\lambda\mu,\lambda\kappa,\mu\kappa\}},
$
\end{lemma}
\begin{proof}
Without loss of generality, we may suppose that $A$ is of rank $3$. By \cref{lem:r4}, the class $\rho_\lambda\rho_\mu\rho_\kappa$ is in the kernel of restriction to $\ker(\lambda\mu\kappa)$, and is therefore divisible by $\rho_{\lambda\mu\kappa}$. The only possibility is that
 \[
\rho_\lambda\rho_\mu\rho_\kappa = \ell\cdot  \rho_{\lambda\mu\kappa}\tau_{\langle\lambda,\mu,\kappa\rangle}^{-1}\tau_\lambda^2\tau_\mu^2\tau_\kappa^2 \beta^{-1} k_{\{1,\lambda\mu,\lambda\kappa,\mu\kappa\}}
\]
for some integer $\ell$. After restriction to $\ker(\lambda\mu)\cap\ker(\mu\kappa)\cong C_2$ this becomes
\[
\rho^3 = \ell\cdot2\rho\tau^2\beta^{-1},
\]
and thus $\ell = 1$ by \cref{lem:cyclic}.
\end{proof}

\begin{prop}\label{lem:unique}
In the situation of \cref{lem:basis}, the class $x$ is unique.
\end{prop}
\begin{proof}
Write $x = \rho_{\lambda_1}\cdots\rho_{\lambda_n}\cdot t\cdot k_{H_1}\cdots k_{H_m}$, and fix another basic generator $x' = \rho_{\lambda_1'}\cdots\rho_{\lambda_{n'}'}\cdot t'\cdot k_{H_1'}\cdots k_{H_{m'}'}$ in the same degree, so that we are claiming $x = x'$. Without loss of generality we may suppose $t = 1$. 

Note that $n=0$ if and only if $n'=0$. Indeed, $n\neq 0$ precisely when $\sigma_\kappa\cdot x = -x$ for some $\kappa$, and likewise $n'\neq 0$ precisely when $\sigma_\kappa\cdot x' = -x'$ for some $\kappa$. As both $x$ and $x'$ generate $\pi_\xi KU_A$, these conditions agree.

Suppose first $n = 0$. Observe $H_1+\cdots+H_m = \{\kappa\in A^\vee : \sigma_\kappa\cdot x = x\}$ and $H_1'+\cdots+H_m' = \{\kappa\in A^\vee : \sigma_\kappa\cdot x' = x'\}$. As both $x$ and $x'$ generate $\pi_\xi KU_A$, it follows that $H_1+\cdots+H_m = H_1'+\cdots+H_m'$. Thus we may suppose without loss of generality that $A = (H_1+\cdots+H_m)^\vee$ is of rank $2m$. In this case $\pi_\xi KU_A = \bbZ\{x\} = \bbZ\{x'\}$, and so $x = \pm x'$. As both $x$ and $x'$ restrict to $2^m$ in $\pi_0 KU$, the only possiblity is that $x = x'$.

Suppose next $n\geq 1$. By a repeated application of \cref{lem:r9}, we may expand $x$ and $x'$ into monomials of the form $x = \rho_{\lambda_1}\cdots \rho_{\lambda_k}\cdot s$ and $x' = \rho_{\lambda_1'}\cdots \rho_{\lambda_k'}\cdot s'$, where $\lambda_1,\ldots,\lambda_k$ are linearly independent,  $\lambda_1',\ldots,\lambda_k'$ are linearly independent, and $s,s'$ are Bott classes. After modifying these by a Bott class we may take $s = 1$. Observe that $\sigma_{\lambda_i'}\cdot x' = -x'$ for $1\leq i \leq k$. As both $x$ and $x'$ generate $\pi_\xi KU_A$, it follows that $\sigma_{\lambda_i'}\cdot x = -x$; thus we may write $\lambda_i' = \lambda_{n_{i,1}}\cdots \lambda_{n_{i,s_i}}$, where $n_{i,1},\ldots,n_{i,s_i}$ are distinct and $s_i$ is odd, and in particular, $\langle\lambda_1',\ldots,\lambda_k'\rangle\subset\langle\lambda_1,\ldots,\lambda_k\rangle$. In the same way we find $\langle\lambda_1,\ldots,\lambda_k\rangle\subset\langle\lambda_1',\ldots,\lambda_k'\rangle$, so these subgroups agree. So we may suppose without loss of generality that $A = \langle\lambda_1,\ldots,\lambda_k\rangle^\vee$ is of rank $k$. In this case $\pi_\xi KU_A = \bbZ\{x\} = \bbZ\{x'\}$, so that $x = \pm x'$, and we must show that this sign is positive. Let $K = \bigcap_{1\leq i < j \leq k}\ker(\lambda_i\lambda_j)$ and write $j\colon K\subset A$ for the inclusion. Write $\lambda$ for the restriction of $\lambda_1$ to $K$, so that $j^\ast(x) = \rho_\lambda^k$. By the decompositions $\lambda_i' = \lambda_{n_{i,1}}\cdots\lambda_{n_{i,s_i}}$, we find that $j^\ast(x') = \rho_\lambda^k \cdot j^\ast(s')$. As $j^\ast(s')$ is a Bott class in $\pi_0 KU_K$, it must be that $j^\ast(s') = 1$, so that $j^\ast(x) = j^\ast(x')$. Thus the sign in $x = \pm x'$ is positive, and $x = x'$.
\end{proof}

\subsection{Relations}\label{ssec:relations}

We must now verify the relations of \cref{sssec:relations}. We begin with those which are by now clear.

\begin{lemma}
\hphantom{blank}
\begin{enumerate}
\item[\ref{r1}] There is at most one basic generator in any single degree;
\item[\ref{r2}] $\rho_\lambda h_\lambda = 0$, or equivalently, $\sigma_\lambda\rho_\lambda = - \rho_\lambda$;
\item[\ref{r3}] $d_{\lambda\mu} = d_\lambda+d_\mu-d_\lambda d_\mu$, or equivalently, $\sigma_{\lambda\mu} = \sigma_\lambda\sigma_\mu$;
\item[\ref{r4}] $\rho_\lambda\rho_\mu\rho_{\lambda\mu}=0$;
\item[\ref{r5}] $\rho_\lambda k_H = 0$ for $\lambda\in H$.
\end{enumerate}
\end{lemma}
\begin{proof}
R.1.~~ This was shown in \cref{lem:unique}.

R.2.~~ This was shown in \cref{lem:cyclic}.

R.3.~~ This follows from $\pi_0 KU_A = RU(A)$ and the definition of the classes involved.

R.4.~~ This was shown in \cref{lem:r4}.

R.5.~~ This holds as the relevant degree vanishes by \cref{lem:basis}, compare \cref{lem:r4}.
\end{proof}

This leaves relations \ref{r6} and \ref{r7}.

\begin{lemma}[\ref{r6}]
$
k_{\langle \lambda,\mu\rangle} k_{\langle \lambda,\kappa\rangle} = 2 \tau_{\langle\lambda,\mu,\kappa\rangle}\tau_{\mu\kappa}^{-2}\tau_{\lambda\mu\kappa}^{-2}k_{\langle \lambda,\mu\kappa\rangle} - \rho_\mu\rho_\kappa\rho_{\lambda\mu}\rho_{\lambda\kappa}\tau_\lambda^2\beta^2.
$
\end{lemma}
\begin{proof}
Without loss of generality we may suppose that $A$ is of rank $3$, so that this product lives in the group $\bbZ\{\tau_{\langle\lambda,\mu,\kappa\rangle}\tau_{\mu\kappa}^{-2}\tau_{\lambda\mu\kappa}^{-2}k_{\langle\lambda,\mu\kappa\rangle}\}\otimes\bbZ\{1,h_\mu\}$. As $k_{\langle\lambda,\mu\rangle}k_{\langle\lambda,\kappa\rangle}$ lifts $4$ in $\pi_0 KU$, and $\rho_\mu\cdot k_{\langle\lambda,\mu\rangle} k_{\langle\lambda,\kappa\rangle} = 0$ by R.5, it follows that
\[
k_{\langle \lambda,\mu\rangle}k_{\langle\lambda,\kappa\rangle} = h_\mu\cdot\tau_{\langle\lambda,\mu,\kappa\rangle}\tau_{\mu\kappa}^{-2}\tau_{\lambda\mu\kappa}^{-2}k_{\langle \lambda,\mu\kappa\rangle}.
\]
This expands out to the more symmetric relation claimed.
\end{proof}

\begin{lemma}[\ref{r7}]\label{lem:ksquare}
$
k_{\langle\lambda,\mu\rangle}^2 = \tau_\lambda^2\tau_\mu^2\tau_{\lambda\mu}^2 h_\lambda h_\mu.
$
\end{lemma}
\begin{proof}
Without loss of generality we may suppose that $A$ is of rank $2$. Now both sides of this equality are the unique class in their degree which lift $4$ in $\pi_0 KU$ and are in the kernel of $\rho_\delta$ for any $\delta\in A^\vee$.
\end{proof}

It must be verified that this is a complete set of relations.

\begin{lemma}\label{lem:link}
Suppose given rank $2$ subgroups $H_1,\ldots,H_m\subset A^\vee$ and $\lambda\in H_1+\cdots+H_m$. Then there are rank $2$ subgroups $H_1',\ldots,H_m'\subset A^\vee$ such that $\lambda\in H_1'$ and $k_{H_1}\cdots k_{H_m} = t \cdot k_{H_1'}\cdots k_{H_m'}$ for a Bott class $t$.
\end{lemma}
\begin{proof}
We induct on $m$, the case $m=1$ being clear. In the inductive step, we may suppose $\lambda\notin H_1+\cdots+\widehat{H_j}+\cdots+H_m$ for any $1\leq j \leq m$, for otherwise the inductive hypothesis already applies. Thus we may write $H_i = \langle \mu_i,\kappa_i\rangle$ in such a way that $\lambda = \mu_1\cdots\mu_m$. Let $H_1'' = \langle \mu_1\mu_2,\kappa_1\rangle$ and $H_2' = \langle \mu_2,\kappa_1\kappa_2\rangle$. Then we have $k_{H_1}k_{H_2} = t'\cdot k_{H_1''}k_{H_2'}$ for a suitable Bott class $t'$ by \ref{r10}. By construction we have $\lambda\in H_1''+H_3+\cdots+H_m$. It follows by induction that $k_{H_1''}k_{H_3}\cdots k_{H_m} = t''\cdot k_{H_1'}h_{H_3'} \cdots k_{H_m'}$ with $\lambda\in H_1'$, and so $H_1',\ldots,H_m'$ satisfy the desired properties.
\end{proof}

\begin{prop}
The above form a complete set of relations, i.e.\
\[
\pi_\star KU_A = \bbZ[\beta^{\pm 1},\tau_\lambda^{\pm 2},\tau_E,\rho_\lambda,k_H]/I,
\]
where $I$ is spanned by relations \ref{r1}--\ref{r7}.
\end{prop}
\begin{proof}
Let us work in the periodic quotient ring of $\pi_\star KU_A$ wherein all Bott elements are identified with $1$; no information is lost in doing so by \ref{r1}. By \cref{lem:basis}, which also incorporates \ref{r1}--\ref{r5}, it is sufficient to verify that the relations in $I$ allow us to write any monomial in the classes $\rho_\lambda$ and $k_H$ as a sum of classes which are a product of some element of $RU(A)$ with a basic generator. So fix some monomial  $x = \rho_{\lambda_1}\cdots\rho_{\lambda_n} k_{H_1}\cdots k_{H_m}$; let us say that such a monomial has $k$-length $m$ and $\rho$-length $n$. If $\lambda_i = \lambda_j$ for some $i\neq j$, then $\rho_{\lambda_i}\rho_{\lambda_j}\in RU(A)$, so we may suppose $\lambda_i\neq \lambda_j$ for $i\neq j$. By a repeated application of \ref{r9}, we may moreover suppose that $x$ has $\rho$-length at most $2$. We now induct on $k$-length without increasing $\rho$-length, splitting into the following cases.

First we claim that if $\lambda_i\in H_1+\cdots+H_m$ for some $i$, then $x = 0$. Indeed, we may suppose that $\lambda_i\in H_1$ by \cref{lem:link}, at which point $x=0$ by \ref{r5}.

Next we claim that if $n=2$ and $\lambda_1\lambda_2\in H_1+\cdots+H_m$, then $x$ is a product of a class in $RU(A)$ with a monomial of smaller $k$-length. Indeed, by \cref{lem:link}, we may suppose $\lambda_1\lambda_2\in H_1$. Write $H_1=\langle\lambda_1\lambda_2,\mu\rangle$. Then $\rho_{\lambda_1}\rho_{\lambda_1\mu}\rho_{\lambda_2\mu} = \rho_{\lambda_2}k_{\langle\lambda_1\lambda_2,\mu\rangle}$ by \ref{r9}, and thus $x = d_{\lambda_1}\rho_{\lambda_1\mu}\rho_{\lambda_2\mu} k_{H_2}\cdots k_{H_m}$, which is of the form claimed.

Finally we claim that if $H_1+\cdots + H_m$ is not of rank $2m$, then $k_{H_1}\cdots k_{H_m}$ may be written as a product of an element of $RU((H_1+\cdots+H_m)^\vee)\subset RU(A)$ with a monomial of smaller $k$-length. Indeed, after possibly rearranging $H_1,\ldots,H_m$, we may suppose $H_1\cap(H_2+\cdots+H_m)\neq 0$; choose $\lambda\neq 1$ in this intersection. Now $\lambda\in H_2+\cdots+H_m$, so by \cref{lem:link} we may suppose $\lambda\in H_2$. The claim now follows by an application of either \ref{r6} or \ref{r7} to the subword $k_{H_1}k_{H_2}$.
\end{proof}

\subsection{Transfers}\label{ssec:transfer}

Fix a codimension $1$ subgroup $\ker(\lambda)\subset A$, and consider the transfer $j_!\colon \pi_{j^\ast\star}KU_{\ker(\lambda)}\rightarrow \pi_\star KU_A$.

\begin{lemma}\label{lem:transfer}
The transfer $j_!$ satisfies the following properties:
\begin{enumerate}
\item[\ref{t1}] $j_!$ is $\pi_\star KU_A$-linear, i.e.\ $j_!(x\cdot j^\ast(y)) = j_!(x)\cdot y$ for $x\in \pi_{j^\ast \star}KU_{\ker(\lambda)}$ and $y\in \pi_\star KU_A$;
\item[\ref{t2}] $j_!\colon \pi_0 KU_{\ker(\lambda)}\rightarrow \pi_0 KU_A$ satisfies $j_!(1) = h_\lambda\in \pi_0 KU_A$;
\item[\ref{t3}] $j_!\colon \pi_{2-2 j^\ast(\mu)}KU_{\ker(\lambda)}\rightarrow\pi_{2-\mu-\lambda\mu}KO_A$ satisfies $j_!(\tau_{j^\ast(\mu)}^2) = \rho_\mu\rho_{\lambda\mu}\beta$;
\item[\ref{t4}] $j_!\colon \pi_{2-2j^\ast(\mu)}KU_{\ker(\lambda)}\rightarrow\pi_{3-\lambda-\mu-\lambda\mu}KU_A$ satisfies $j_!(\tau_{j^\ast(\mu)}^2) = k_{\langle\lambda,\mu\rangle}$.
\end{enumerate}
\end{lemma}
\begin{proof}
T.1.~~ This is a general property of transfers.

T.2.~~ This follows from the definition of $h_\lambda = 1 + \sigma_\lambda$.

T.3. Without loss of generality we may suppose that $A$ is of rank $2$. Write $\sigma = j^\ast(\mu)$. As $\rho_\mu\rho_{\lambda\mu}\beta$ is in the kernel of $\rho_\lambda$, it is in the image of $j_!$, and thus $j_!(\tau_\sigma^2) = \pm \rho_\mu\rho_{\lambda\mu}\beta$. We must show that this sign is positive. By $\pi_\star KU_A$-linearity, we may compute $j_!(\tau_\sigma^2 d_\sigma) = j_!(\tau_\sigma) d_\mu = \pm \rho_\mu \rho_{\lambda\mu}\beta\cdot d_\mu = \pm 2 \rho_\mu \rho_{\lambda\mu}\beta$, and thus $j^\ast j_!(\tau_\sigma^2 d_\sigma) = \pm 2 \tau_\sigma^2 d_\sigma$, this $\pm$ agreeing with the previous. On the other hand, let $Q$ be the generator of $A/\ker(\lambda)\cong C_2$. Then the double coset formula yields $j^\ast j_!(\tau_\sigma^2 d_\sigma) = \tau_\sigma^2 d_\sigma + Q(\tau_\sigma^2 d_\sigma)$. For $\tau_\sigma^2 d_\sigma + Q(\tau_\sigma^2 d_\sigma) = \pm 2 \tau_\sigma^2 d_\sigma$ to hold with $Q$ an involution, the only possibility is that $Q(\tau_\sigma^2 d_\sigma) = \tau_\sigma^2 d_\sigma$, so the relevant sign is positive.

T.4.~~ This was shown in \cref{lem:halftransfer}.
\end{proof}

We must verify that these properties fully determine $j_!$.

\begin{lemma}\label{lem:lambdanomial}
Fix a nontrivial functional $\lambda\in A^\vee$. Then any basic generator may be represented by a basic monomial of the form $x = \rho_{\lambda_1}\cdots\rho_{\lambda_n}\cdot t\cdot k_{H_1}\cdots k_{H_m}$ satisfying one of the following conditions:
\begin{enumerate}
\item $\lambda\notin\langle\lambda_1,\ldots,\lambda_n\rangle+H_1+\cdots+H_m$;
\item $\lambda = \lambda_1$;
\item $\lambda = \lambda_1\lambda_2$;
\item $\lambda\in H_1$.
\end{enumerate}
\end{lemma}
\begin{proof}
Fix an arbitrary basic generator $x=\rho_{\lambda_1}\cdots\rho_{\lambda_n}\cdot t\cdot k_{H_1}\cdots k_{H_m}$ with $n\leq 2$, and suppose that none of (1)--(4) hold. We are then left with the following possibilities.

First suppose $\lambda \in H_1+\cdots+H_m$. By \cref{lem:link} we may suppose $\lambda\in H_1$, reducing us to case (4).

Next suppose $n=1$ and $\lambda\in \langle\lambda_1\rangle+H_1+\cdots+H_m$. By \cref{lem:link}, we may suppose $H_1 = \langle\lambda\lambda_1,\kappa\rangle$. Now $\rho_{\lambda_1}k_{\langle\lambda\lambda_1,\kappa\rangle} = \rho_{\lambda}\rho_{\lambda_1\kappa}\rho_{\lambda\kappa}\cdot t'$ for a Bott class $t'$ by \ref{r9}, putting us in case (2).

Finally suppose $n=2$ and $\lambda\in\langle\lambda_1,\lambda_2\rangle+H_1+\cdots+H_m$. By the preceding case and \cref{lem:link}, we may suppose $\lambda = \lambda_1\lambda_2\mu$ with $\mu\in H_1$. Write $H_1 = \langle\lambda\lambda_1\lambda_2,\kappa\rangle$. Now $\rho_{\lambda_1}\rho_{\lambda_2}k_{\langle \lambda\lambda_1\lambda_2,\kappa\rangle} = \rho_{\lambda_1\kappa}\rho_{\lambda_2\kappa}k_{\langle\lambda,\kappa\rangle}\cdot t'$ for a Bott class $t'$ by \ref{r11}, putting us in case (4).
\end{proof}

\begin{prop}
The transfer $j_!$ is determined by the properties given in \cref{lem:transfer}.
\end{prop}

\begin{proof}
Fix $\xi\in RO(A)$; we must verify that $j_!\colon \pi_{j^\ast\xi}KU_{\ker(\lambda)}\rightarrow\pi_\xi KU_A$ may be computed from the given properties. If $\pi_\xi KU_A = 0$, then there is nothing to show, so we may suppose that $\pi_\xi KU_A$ contains some basic monomial $x$ of the form described in \cref{lem:lambdanomial}. Applying \ref{t1}, we may focus our attention on only those subwords which interact with $\lambda$, and so reduce to the following cases.

If $x = 1$, then we may apply \ref{t2}.

If $x = \rho_\lambda$, then $\pi_{j^\ast(\xi)}KU_{\ker(\lambda)} = 0$, and there is nothing to show.

If $x = \rho_{\mu}\rho_{\lambda\mu}$, then $\pi_{j^\ast\xi}KU_{\ker(\lambda)}$ is generated by $j^\ast(\tau^2_{\mu}\beta^{-1})$, and $j_!(j^\ast(\tau^2_{\mu}\beta^{-1})) = j_!(j^\ast(\tau^2_{\mu}))\cdot\beta^{-1} = \rho_{\mu}\rho_{\lambda\mu}\beta\cdot\beta^{-1} = x$ by \ref{t3}.

If $x = k_{\langle\lambda,\mu\rangle}$, then $\pi_{j^\ast\xi}KU_{\ker(\lambda)}$ is generated by $j^\ast(\tau^2_\mu)$, and $j_!(j^\ast(\tau^2_\mu)) = k_{\langle\lambda,\mu\rangle} = x$ by \ref{t4}.
\end{proof}

\subsection{Power operations}\label{ssec:squaring}

Let $\sigma$ be the generating functional of $C_2$, and write $j\colon A\rightarrow A\times C_2$ for the inclusion. Here we compute the external squaring operation
\[
\Sq\colon \pi_\star KU_A\rightarrow \pi_{\star(1+\sigma)} KU_{A\times C_2}
\]
on the multiplicative generators of $\pi_\star KU_A$.

\begin{lemma}\label{lem:squarebott}
$\Sq$ preserves Bott classes.
\end{lemma}
\begin{proof}
First we claim $\Sq(\beta) = \tau_\sigma^{-2}\beta^2$.  Let $L$ be the tautological complex line bundle over $S^2$, so that $\beta = 1 - L \in \widetilde{KU}_A(S^2)$. By construction \cite{atiyah1966power}, the square $\Sq(\beta)$ is represented by the virtual bundle $(1-L)\otimes (1-L) = 1 - (L\oplus L) + L \otimes L$, where $C_2\subset A\times C_2$ acts freely on $L\oplus L$ and by a sign on $L\otimes L$. On the other hand, $\tau_\sigma^{-2}\beta^2$ is the Bott class of $L\otimes \bbC[C_2]$, which is given by the exterior algebra $\Lambda^\ast (L\otimes \bbC[C_2]) = 1 - L \otimes \bbC[C_2] + \Lambda^2(L\otimes \bbC[C_2])$. These agree, so $\Sq(\beta) = \tau_\sigma^{-2}\beta^2$ indeed. The same argument may be used to verify that $\Sq(\tau_\lambda^{-2}\beta) = \tau_\lambda^{-2}\tau_{\lambda\sigma}^{-2}\beta^2$, and thus $\Sq(\tau_\lambda^2) = \tau_\lambda^2\tau_{\lambda\sigma}^2\tau_\sigma^{-2}$.

To verify that $\Sq(\tau_E)$ is a Bott class, we may argue as follows. Let $\xi = (8-\sum_{\lambda\in E}\lambda)(1+\sigma)$, and let $t$ be the Bott class of $\xi$, so that $\pi_\xi KU_{A\times C_2} = \bbZ\{t\}\otimes RU(A\times C_2)$ and we are claiming $\Sq(\tau_E) = t$. The joint restriction map
\[
\pi_\xi KU_{A\times C_2}\rightarrow \prod_{\substack{i\colon L\subset A \\ L\text{ cyclic}}} \pi_{(i\times C_2)^\ast \xi}KU_{L\times C_2}
\]
is injective, so it is sufficient to fix some inclusion $i\colon C_2\rightarrow A$ and verify that $(i\times C_2)^\ast(\Sq(\tau_E)) = (i\times C_2)^\ast(t)$. Indeed, $(i\times C_2)^\ast(\Sq(\tau_E)) = \Sq(i^\ast \tau_E)$, and $i^\ast(\tau_E)$ is a product of complex Bott classes, so this follows from the cases already considered.
\end{proof}

\begin{lemma}
$\Sq(\rho_\lambda) = \rho_\lambda\rho_{\lambda\sigma}$.
\end{lemma}
\begin{proof}
This is the only possibility given $j^\ast \Sq(\rho_\lambda)=\rho_\lambda^2$.
\end{proof}

\begin{lemma}
$
\Sq(k_{\langle\lambda,\mu\rangle}) = \tau_{\langle\lambda,\mu,\sigma\rangle}\tau_\sigma^{-4}(\sigma_\lambda+\sigma_\mu+\sigma_{\lambda\mu}+\sigma_\sigma).
$
\end{lemma}
\begin{proof}
Note that
\[
\Sq(k_{\langle\lambda,\mu\rangle}) \in \pi_{(3-\lambda-\mu-\lambda\mu)(1+\sigma)}KU_{A\times C_2}= \bbZ\{\tau_{\langle\lambda,\mu,\sigma\rangle}\tau_\sigma^{-4}\}\otimes RU(A\times C_2).
\]
This class depends only on the group $\langle\lambda,\mu\rangle$, so is of the form
\[
\Sq(k_{\langle\lambda,\mu\rangle}) =\tau_{\langle\lambda,\mu,\delta\rangle}  \tau_\sigma^{-4}(a+b(\sigma_\lambda+\sigma_\mu+\sigma_{\lambda\mu})+c\sigma_\sigma+d(\sigma_\lambda\sigma_\sigma+\sigma_\mu\sigma_\sigma+\sigma_{\lambda\mu}\sigma_\sigma))
\]
for some integers $a,b,c,d$. As $\Sq(k_{\langle\lambda,\mu\rangle})$ restricts to $k_{\langle\lambda,\mu\rangle}^2 = \tau_\lambda^2\tau_\mu^2\tau_{\lambda\mu}^2(\sigma_\lambda+\sigma_\mu+\sigma_{\lambda\mu}+1)$ over $A$ and to $\Sq(2)=3+\sigma_\sigma$ over $C_2$, these integers satisfy
\[
a+b=1,\qquad b+d=1,\qquad a+3b=3,\qquad c+3d=1.
\]
This system has the unique solution $a=d=0$ and $b=c=1$, and the lemma follows.
\end{proof}

This concludes our computation of $\pi_\star KU_A$.

\section{Real \texorpdfstring{$K$}{K}-theory}\label{sec:ko}

We now consider the descent to $KO_A$. Throughout this section, we shall write
\[
\theta\colon \pi_\star KO_A\rightarrow\pi_\star KU_A
\]
for the complexification map.

\subsection{Summary}\label{ssec:kosummary}
As with $KU_A$, we begin with a full description of the result.

\begin{theorem}\label{thm:koa}
The coefficients of $KO_A$ behave as described in this subsection.
\tqed
\end{theorem}

The proof of \cref{thm:koa} is spread throughout the rest of this section, glued together as described below. The core of the proof is the homotopy fixed point spectral sequence
\[
E_2 = H^\ast(C_2;\pi_\star KU_A)\Rightarrow \pi_\star KO_A,
\]
henceforth referred to as the HFPSS, obtained from the equivalence $KO_A\simeq (KU_A)^{\h C_2}$, where $C_2$ acts on $KU_A$ by complex conjugation, realized by $\psi^{-1}$.

\subsubsection{Ring structure}

We shall name the elements of $\pi_\star KO_A$ by their image in $\pi_\star KU_A$, with the following exceptions. First, we write $\alpha\in \pi_1 KO_A$ for the first nonequivariant Hopf map. Second, we abbreviate $\tau_H = \prod_{\lambda\in H\setminus \{1\}}\tau_\lambda^2$, where as always $H\subset A^\vee$ is a rank $2$ subgroup. Third, we write $\eta_\lambda\in \pi_\lambda KO_A$ for a class determined by $\theta(\eta_\lambda) = \rho_\lambda \tau_\lambda^{-2}\beta$. 
The ring $\pi_\star KO_A$ is now described by the following.

\begin{enumerate}
\item The ring $\pi_\star KO_A$ is generated by classes
\[
\beta^{\pm 4},\,2\beta^2,\,\tau_\lambda^{\pm 4},\,\tau_H,\,\tau_E,\,\rho_\lambda,\,\eta_\lambda,\,\tau^2_\lambda k_H,\,\beta^2 k_H,\,2k_H,\,2\tau^2_\lambda\beta^2 k_H,\,\tau^2_\lambda h_\lambda,\,\tau^2_\lambda\beta^2 h_\lambda,\,\alpha,
\]
which are sent by $\theta$ to the corresponding elements in $\pi_\star KU_A$, where in writing $\tau^2_\lambda k_H$ and $2\tau^2_\lambda\beta^2 k_H$ we assume $\lambda\in H$;
\item The map $\theta\colon (\pi_\star KO_A)/(\alpha)\rightarrow \pi_\star KU_A$ is injective;
\item The following classes vanish:
\[
2\alpha,\quad \alpha^3,\quad \alpha\cdot 2\beta^2,\quad \alpha\cdot 2k_H,\quad \alpha\cdot 2\tau^2_\lambda\beta^2 k_H;
\]
\item The following relations hold:
\begin{gather*}
\rho_\lambda\rho_\mu\rho_{\lambda\mu} = \beta^{-2}k_{\langle\lambda,\mu\rangle}\cdot\alpha,\qquad \rho_\lambda\rho_\mu\eta_{\lambda\mu} = 0,\qquad \rho_{\lambda\mu}\eta_\lambda\eta_\mu = \tau_\lambda^{-2}\tau_\mu^{-2} k_{\langle\lambda,\mu\rangle}\cdot\alpha,\qquad \eta_\lambda\eta_\mu\eta_{\lambda\mu} = 0,\\
\rho_\lambda\cdot \tau^2_\lambda h_\lambda = 0,\qquad \eta_\lambda\cdot\tau^2_\lambda h_\lambda = \rho_\lambda\alpha^2,\qquad \rho_\lambda\cdot\tau_\lambda^2\beta^2 h_\lambda = \eta_\lambda\tau_\lambda^4\alpha^2,\qquad \eta_\lambda\cdot\tau_\lambda^2\beta^2h_\lambda = 0,\\
\rho_{\lambda\mu}\cdot\tau_{\lambda\mu}^2 k_{\langle\lambda,\mu\rangle} = \rho_\lambda\rho_\mu\tau_{\lambda\mu}^4\alpha,\qquad \rho_\lambda\cdot\tau_\mu^2k_{\langle\lambda,\mu\rangle} = 0,\qquad \rho_{\lambda\mu}\cdot\beta^2 k_{\langle\lambda,\mu\rangle} = \eta_\lambda\eta_\mu\tau_{\langle\lambda,\mu\rangle}\alpha, \\
\eta_{\lambda\mu}\cdot\tau^2_{\lambda\mu}k_{\langle\lambda,\mu\rangle} = 0,\qquad \eta_\lambda\cdot\tau_{\lambda\mu}^2k_{\langle\lambda,\mu\rangle} = \rho_\mu \eta_{\lambda\mu}\tau_{\lambda\mu}^{4}\alpha,\qquad \eta_{\lambda\mu}\cdot \beta^2k_{\langle\lambda,\mu\rangle} = 0.
\end{gather*}
\end{enumerate}
This computation will be carried out in \cref{ssec:hfpss} and \cref{ssec:extensions}.

\begin{rmk}
The products in (4) which vanish do so for degree reasons. This leads to the simpler rule: if an extension may exist, then the extension does exist.
\tqed
\end{rmk}

\begin{rmk}
Write $\sigma$ for the generating functional of $C_2$. Then $\eta_\sigma = -\eta_{C_2}$, where $\eta_{C_2}$ is the $C_2$-equivariant Hopf map with conventions as in e.g.\ \cite{guillouhillisaksenravenel2020cohomology}.
\tqed
\end{rmk}

\subsubsection{Basis}\label{sssec:kobasis}

Fix $\xi \in RO(A)$. Then $\pi_{\xi+\ast}KO_A$ is either a free $KO_\ast$-module or a direct sum of copies of $KU_\ast$. In the former case, $\pi_{\xi+\ast}KO_A$ is generated over $KO_\ast\otimes RO(A)$ by a class of the form
\[
x = \rho_{\lambda_1}\cdots\rho_{\lambda_n}\cdot\eta_{\mu_1}\cdots\eta_{\mu_s}\cdot t \cdot\beta^2 k_{H_1}\cdots \beta^2 k_{H_m}\cdot \tau^2_{\kappa_1}k_{H_{m+1}}\cdots\tau^2_{\kappa_t}k_{H_{m+t}},
\]
where
\begin{enumerate}
\item $\lambda_1,\ldots,\lambda_n,\mu_1,\ldots,\mu_s$ are linearly independent, and one may suppose $n,s,n+s\leq 2$;
\item $t$ is a product of classes of the form $\beta^{\pm 4}$, $\tau_\lambda^{\pm 4}$, $\tau_H$, $\tau_E$;
\item $H_1+\cdots+H_{m+t}$ is of rank $2(m+t)$;
\item $\kappa_i\in H_{m+i}$ for $1\leq i \leq t$;
\item $\langle\lambda_1,\ldots,\lambda_n,\mu_1,\ldots,\mu_s\rangle\cap (H_1+\cdots+H_{m+t}) = 0$.
\end{enumerate}
In the latter case, $\pi_{\xi+\ast}KO_A$ may be regarded as a $KU_\ast\otimes RO(A)$-module, and is generated by a class of the form $x\cdot \tau^2h_\delta$ where $x$ is as above and $\delta\notin \langle\lambda_1,\ldots,\lambda_n,\mu_1,\ldots,\mu_s\rangle+H_1+\cdots+H_{m+t}$. In either case, such classes are unique in their degree, though their presentation as a monomial need not be.

All of this follows from analogous statements for $KU_A$ (\cref{sssec:basis}) and the work of \cref{ssec:hfpss}.

\subsubsection{Mackey structure}\label{sssec:komackey}

Fix a second elementary abelian $2$-group $B$, and map $g\colon A\rightarrow B$. The restriction
\[
g^\ast\colon \pi_\star KO_B\rightarrow \pi_{g^\ast\star}KO_A
\]
is determined by the following.
\begin{enumerate}
\item $g^\ast$ commutes with $\theta$;
\item $g^\ast(\alpha) = \alpha$;
\item $g^\ast(\eta_\lambda) = \eta_{g^\ast\lambda}$, with the interpretation that $\eta_1 = \alpha$.
\end{enumerate}
Here, (1) and (2) are clear, and we will verify (3) in \cref{lem:eta}.

Now fix a codimension $1$ subgroup $j\colon \ker(\lambda)\rightarrow A$, inducing a transfer
\[
j_!\colon \pi_{j^\ast\star}KO_{\ker(\lambda)}\rightarrow \pi_{\star}KO_A.
\]
This is determined by the following.
\begin{enumerate}
\item $j_!$ commutes with $\theta$;
\item $j_!$ is $\pi_\star KO_A$-linear;
\item $j_!\colon \pi_0 KO_{\ker(\lambda)}\rightarrow \pi_{1-\lambda} KO_A$ satisfies $j_!(1) = \rho_\lambda\alpha$.
\end{enumerate}
We will verify this in \cref{ssec:kotransfer}.

The Weyl action is formally determined by these as in \cref{sssec:weyl}. 

\subsubsection{Operations}

As with $KU_A$, there is an external squaring operation
\[
\Sq\colon \pi_\star KO_A\rightarrow \pi_{\star(1+\sigma)}KO_{A\times C_2},
\]
where we have written $\sigma$ for the generating functional of $C_2$. This commutes with $\theta$, satisfies the identities
\[
\Sq(xy) = \Sq(x)\Sq(y),\qquad \Sq(x+y) = \Sq(x)+\Sq(y)+\tr(xy),
\]
where $\tr$ is the transfer, and is otherwise determined by
\[
\Sq(\alpha) = \eta_\sigma \alpha.
\]
Indeed this is the only class in its degree that lifts $\alpha^2$.

Finally, fix an integer $\ell$, so that the Adams operation $\psi^\ell$ acts on $\pi_\star KO_A[\tfrac{1}{\ell}]$ by ring automorphisms. This commutes with $\theta$, and is otherwise determined by 
\[
\psi^\ell(\alpha)=\alpha.
\]
This is clear, as $\alpha$ is in the Hurewicz image.

\subsection{The HFPSS}\label{ssec:hfpss}

We begin by computing the HFPSS
\[
E_2 = H^\ast(C_2;\pi_\star KU_A)\Rightarrow \pi_\star KO_A.
\]

\begin{lemma}
The subring
\[
H^0(C_2;\pi_\star KU_A)\subset \pi_\star KU_A
\]
is generated by the following elements:
\[
\beta^{\pm 2},~ \tau_\lambda^{\pm 4},~ \tau_H^2,~ \tau_E,~ \rho_\lambda,~ \eta_\lambda,~ k_H,~ \tau_\lambda^2 k_H,~ \tau^2_\lambda h_\lambda.
\]
Here, in writing $\tau_\lambda^2 k_H$ we assume $\lambda\in H$. Where $\alpha$ generates $H^1(C_2;\bbZ\{\beta\})$, we have
\[
H^\ast(C_2;\pi_\star KU_A) = H^0(C_2;\pi_\star KU_A)[\alpha]/(2\alpha,\, \rho_\lambda^2 \cdot \alpha,\, \tau^2_\lambda h_\lambda\cdot\alpha).
\]
\end{lemma}\begin{proof}
Note first $H^\ast(C_2;\pi_\ast KU) = \bbZ[\beta^{\pm 2},\alpha]/(2\alpha)$, and that $\pi_0 KU_A$ is entirely fixed by $\psi^{-1}$. Fix a basic monomial
\[
x = \rho_{\lambda_1}\cdots\rho_{\lambda_n} \cdot t \cdot k_{H_1}\cdots k_{H_m}
\]
such that $t$ is a product of classes of the form $\tau_\lambda^2$ and $\tau_E$. It is sufficient to verify the following: if $\psi^{-1}(x) = x$, then $x$ is a product of the listed generators; if $\psi^{-1}(x) = -x$, then $\beta x$ is a product of the listed generators;  and finally if $\psi^{-1}(x)$ is linearly independent from $x$, then both $x + \psi^{-1}(x)$ and $\beta^{-1} x + \psi^{-1}(\beta^{-1} x)$ are products of the listed generators, this product involves either some $\rho_\lambda^2$ or $\tau^2_\lambda h_\lambda$, and both $\rho_\lambda^2$ and $\tau^2_\lambda h_\lambda$ may be obtained as such a class.

As $\tau_E$ and $\tau_\lambda^{\pm 4}$ are fixed by $\psi^{-1}$, we may suppose that $t$ is of the form $t = \tau_{\mu_1}^2\cdots\tau_{\mu_s}^2$. If $s\geq 2$, then we may inductively apply the relation $\tau_{\mu_1}^2\cdots\tau_{\mu_s}^2 = \tau_{\langle\mu_1,\mu_2\rangle} \cdot \tau_{\mu_1\mu_2}^{-4} \cdot \tau_{\mu_1\mu_2}^2\tau_{\mu_3}^2\cdots\tau_{\mu_s}^2$ to further reduce to the case where $t = 1$ or $t = \tau_\mu^2$. In the former case, $x$ is fixed by $\psi^{-1}$ and is a product of the listed generators, so consider the latter case.

Suppose first $\mu \in \langle\lambda_1,\ldots,\lambda_n\rangle + H_1+\cdots+H_m$. After possibly reordering $\lambda_1,\ldots,\lambda_n$ and $H_1,\ldots,H_m$, we may suppose $\mu = \lambda_1\cdots\lambda_r\cdot\kappa_1\cdots\kappa_s$ with $0\leq r \leq n$, $0\leq s\leq m$, and $\kappa_i\in H_i$. We now have
\[
x = \eta_{\lambda_1}\cdots\eta_{\lambda_r}\cdot \rho_{\lambda_{r+1}}\cdots\rho_{\lambda_n}\cdot\beta^{-r}\cdot \tau^2_{\kappa_1} k_{H_1}\cdots\tau^2_{\kappa_s}k_{H_s}\cdot k_{H_{s+1}}\cdots k_{H_m}\cdot \tau_{\lambda_1}^{-2}\cdots\tau_{\lambda_r}^{-2}\cdot\tau_{\kappa_1}^{-2}\cdots\tau_{\kappa_r}^{-2}\cdot\tau_{\mu}^2.
\]
If $r$ is even, then this is fixed by $\psi^{-1}$, and is a product of the listed generators, and if $r$ is odd then the same is true of $\beta x$.

Suppose next $\mu\notin \langle\lambda_1,\ldots,\lambda_n\rangle + H_1+\cdots+H_m$. In this case we have
\begin{gather*}
x + \psi^{-1}(x) = \rho_{\lambda_1}\cdots\rho_{\lambda_n}\cdot\tau^2_\mu h_\mu \cdot k_{H_1}\cdots k_{H_m} \\
\beta^{-1} x + \psi^{-1}(\beta^{-1} x) = \rho_{\lambda_1}\cdots\rho_{\lambda_n}\cdot\rho_\mu^2\cdot k_{H_1}\cdots k_{H_m},
\end{gather*}
and these satisfy the desired properties.
\end{proof}

\begin{lemma}\label{lem:diffs}
The differentials in the HFPSS are determined by
\begin{gather*}
d_3(\beta^2) = \alpha^3,\qquad d_3(\tau_\lambda^4) = 0,\qquad d_3(\tau_H^2)=0,\qquad d_3(\tau_E) = 0,\qquad d_3(\rho_\lambda) = 0\\
 d_3(\eta_\lambda) = 0, \qquad d_3(\beta^2k_H) = 0,\qquad d_3(\tau_\lambda^2 k_H) = 0, \qquad d_3(\tau_\lambda^2 h_\lambda) = 0,
\end{gather*}
after which $E_4 = E_\infty$.
\end{lemma}
\begin{proof}
The differential $d_3(\beta^2) = \alpha^3$ is standard. The structure of $H^\ast(C_2;\pi_\star KU_A)$ then implies that for each multiplicative generator $x$, either $d_3(x) = 0$ or $d_3(x) = \beta^{-2} x \alpha^3$, and that these are the only differentials. Now $\tau_\lambda^4$, $\tau_H^2$, and $\tau_E$ are cycles as they are Thom classes of $\Spin$ bundles, and $\rho_\lambda$, $\eta_\lambda$, $\tau^{-2}_\lambda k_H$, and $\tau_\lambda^2 h_\lambda$ are cycles as they are in the Hurewicz image, the first by construction, second by its relation to the equivariant Hopf map, and last two as they are of the form $\tr(1)$. It remains to show that $k_H$ is not a cycle, and here may suppose without loss of generality that $A^\vee = H = \langle\lambda,\mu\rangle$.

Recall that $k_H$ restricts to $\tau^2 h$ over each of $\ker(\lambda)$, $\ker(\mu)$, and $\ker(\lambda\mu)$, and that this class is killed by $\alpha$. Thus, if $d_3(k_H) = 0$ then $k_H\cdot \alpha$ survives to a class which is divisible by each of $\rho_\lambda$, $\rho_\mu$, and $\rho_{\lambda\mu}$, and if instead $d_3(\beta^{-2} k_H) = 0$ then the same holds for $\beta^{-2} k_H \cdot \alpha$. In either case $\rho_\lambda\rho_\mu\rho_{\lambda\mu} \neq 0$, and the only possibility is that $\rho_\lambda\rho_\mu\rho_{\lambda\mu} = \beta^{-2}k_H \cdot \alpha$, so it must be that $\beta^{-2}k_H$ is a cycle.
\end{proof}

\subsection{Extensions}\label{ssec:extensions}

There is room for hidden extensions in the HFPSS, and to fully describe $\pi_\star KO_A$ we must resolve these. Our work is simplified by the following observation: in any given stem, the $E_\infty$ page of the HFPSS is concentrated in a single filtration. In particular, there is no room for nontrivial additive extensions, and no room for hidden multiplicative extensions with additional indeterminacy. Thus there are three basic relations in $\pi_\star KU_A$ we must consider:
\[
\rho_\lambda h_\lambda = 0,\qquad \rho_\lambda\rho_\mu\rho_{\lambda\mu} = 0,\qquad \rho_\lambda k_H = 0,
\]
the last assuming $\lambda\in H$. The relations on the $E_\infty$ page of the HFPSS which may hide a nontrivial product in $\pi_\star KO_A$ are of this form, only where $\eta_\kappa$ may take the place of $\rho_\kappa$, where $\tau^2_\kappa h_\kappa$ or $\tau^2_\kappa\beta^2 h_\kappa$ must take the place of $h_\kappa$, and where $\tau^2_\kappa k_H$ or $\beta^2 k_H$ must take the place of $k_H$. This reduces our work to a case analysis. Before carrying this out, we note the following.

\begin{lemma}\label{lem:eta}
Let $g\colon A\rightarrow B$ be a map of elementary abelian $2$-groups. Then
\[
g^\ast(\eta_\lambda)=\eta_{g^\ast\lambda},
\]
with the interpretation that $\eta_1=\alpha$.
\end{lemma}
\begin{proof}
We need only consider the case where $g^\ast\lambda = 1$, and here we may reduce to the case of $g\colon e\rightarrow C_2$. Write $\sigma$ for the generating functional of $C_2$. As $\eta_\sigma$ is not in the image of $\rho_\sigma$, it must have nonzero image in $\pi_1 KO$, so must be $\alpha$.
\end{proof}

We may now proceed to our case analysis. Half of these cases proceed by observing that $\pi_\star KO_A$ vanishes in the degree containing the product under consideration. We illustrate these in \cref{lem:rel1}(2), omitting the analogous details in the remaining cases.

\begin{lemma}\label{lem:rel1}
\hphantom{blank}
\begin{enumerate}
\item $\rho_\lambda\rho_\mu\rho_{\lambda\mu} = \beta^{-2}k_{\langle\lambda,\mu\rangle}\cdot\alpha$.
\item $\rho_\lambda\rho_\mu\eta_{\lambda\mu} = 0$.
\item $\rho_{\lambda\mu}\eta_\lambda\eta_\mu = \tau_\lambda^{-2}\tau_\mu^{-2}k_{\langle\lambda,\mu\rangle}\cdot\alpha$.
\item $\eta_\lambda\eta_\mu\eta_{\lambda\mu}=0$.
\end{enumerate}
\end{lemma}
\begin{proof}
We may suppose without loss of generality that $A$ is of rank $2$.

(1)~~ The class $\beta^{-2} k_{\langle\lambda,\mu\rangle}\cdot\alpha$ is in the kernel of restriction to each of $\ker(\lambda)$, $\ker(\mu)$, and $\ker(\lambda\mu)$. It is therefore divisible by each of $\rho_\lambda$, $\rho_\mu$, and $\rho_{\lambda\mu}$, and this is the only possibility.

(2)~~ This holds as $\pi_{-1+(3-\lambda-\mu-\lambda\mu)-(2-2\lambda\mu)}KO_A = 0$. To see this, first observe that $\pi_{(3-\lambda-\mu-\lambda\mu)-(2-2\lambda\mu)}KU_A = \bbZ\{\tau_{\lambda\mu}^{-2} k_{\langle \lambda,\mu\rangle}\}$. It follows from \cref{lem:diffs} that the $\bbZ$-graded piece $\pi_{\ast+(3-\lambda-\mu-\lambda\mu)-(2-2\lambda\mu)}KO_A$ is a free $KO_\ast$-module generated by $\tau_{\lambda\mu}^{-2}k_{\langle\lambda,\mu\rangle}$. That $\pi_{-1+(3-\lambda-\mu-\lambda\mu)-(2-2\lambda\mu)}KO_A = 0$ then follows as $\pi_{-1}KO = 0$.

(3)~~ The class $\tau_\lambda^{-2}\tau_\mu^{-2} k_{\langle\lambda,\mu\rangle}\cdot\alpha$ is in the kernel of restriction to $\ker(\lambda\mu)$, and is therefore divisible by $\rho_{\lambda\mu}$. This is the only possibility.

(4)~~ This holds as $\pi_{3+(3-\lambda-\mu-\lambda\mu)-(2-2\lambda)-(2-2\mu)-(2-2\lambda\mu)} KO_A = 0$.
\end{proof}

\begin{lemma}\label{lem:c2hidden}
\hphantom{blank}
\begin{enumerate}
\item $\rho_\lambda\cdot\tau^2_\lambda h_\lambda = 0$;
\item $\eta_\lambda\cdot \tau^2_\lambda h_\lambda = \rho_\lambda\alpha^2$;
\item $\rho_\lambda\cdot \tau^2_\lambda\beta^2 h_\lambda =  \eta_\lambda\tau_\lambda^4\alpha^2$;
\item $\eta_\lambda\cdot \tau^2_\lambda \beta^2 h_\lambda = 0$.
\end{enumerate}
\end{lemma}
\begin{proof}
(1)~~ This holds as $\pi_{2-3\lambda} KO_A = 0$.

(2)~~ Without loss of generality we may suppose that $A$ is of rank $2$. Choose $\mu$ linearly independent from $\lambda$, write $j\colon\ker(\lambda)\rightarrow A$ for the inclusion, and write $\sigma = j^\ast(\lambda)$. The class $\rho_\lambda\rho_{\lambda\mu}\alpha^2$ is in the kernel of $\rho_\mu$, and thus in the image of $j_!$, and the only possibility is that $j_!(\tau^2_\sigma h_\sigma) = \rho_\lambda\rho_{\lambda\mu}\alpha^2$. On the other hand, by comparison with $KU_A$ we may compute $j_!(\tau^2_\sigma h_\sigma) = j_!(1)\cdot \tau^2_\lambda h_\lambda =\rho_{\lambda\mu}\eta_\lambda \cdot \tau^2_\lambda h_\lambda$. It follows that $\eta_\lambda\cdot \tau_\lambda^2 h_\lambda \neq 0$, and the indicated relation is the only possibility.

(3)~~ The class $ \eta_\lambda\tau_\lambda^4 \alpha^2$ restricts to $\alpha^3 = 0$ over $\ker(\lambda)$, and is thus in the image of $\rho_\lambda$. The indicated relation is the only possibility.

(4)~~ This holds as $\pi_{6-\lambda} KO_A = 0$.
\end{proof}

\begin{lemma}
\hphantom{blank}
\begin{enumerate}
\item $\rho_{\lambda\mu}\cdot\tau_{\lambda\mu}^2 k_{\langle\lambda,\mu\rangle} = \rho_{\lambda}\rho_\mu\tau_{\lambda\mu}^4\alpha$;
\item $\rho_\lambda\cdot\tau_\mu^2 k_{\langle\lambda,\mu\rangle} = 0$;
\item $\rho_{\lambda\mu}\cdot\beta^2k_{\langle\lambda,\mu\rangle} =\eta_\lambda\eta_{\mu} \tau_{\langle\lambda,\mu\rangle} \alpha$;
\item $\eta_{\lambda\mu}\cdot\tau_{\lambda\mu}^2k_{\langle\lambda,\mu\rangle} = 0$;
\item $\eta_\lambda\cdot\tau_{\lambda\mu}^2k_{\langle\lambda,\mu\rangle} = \rho_\mu \eta_{\lambda\mu}\tau_{\lambda\mu}^{4}\alpha$;
\item $\eta_\lambda\cdot\beta^2 k_{\langle\lambda,\mu\rangle} = 0$.
\end{enumerate}
\end{lemma}
\begin{proof}
We may suppose without loss of generality that $A$ is of rank $2$.

(1)~~ The class $\rho_\lambda\rho_\mu \tau_{\lambda\mu}^4\alpha$ is in the kernel of restriction to $\ker(\lambda\mu)$, and is therefore divisible by $\rho_{\lambda\mu}$. This is the only possibility.

(2)~~ This holds as $\pi_{1+(2-2\lambda)+(2-2\mu)-\mu-\lambda\mu}KO_A = 0$.

(3)~~ The class $\eta_\lambda\eta_{\mu}\tau_{\langle\lambda,\mu\rangle}\alpha$ is in the kernel of restriction to $\ker(\lambda\mu)$, and is therefore divisible by $\rho_{\lambda\mu}$. This is the only possibility.

(4)~~ This holds as $\pi_{3+(2-2\lambda\mu)-\lambda-\mu} KO_A = 0$.

(5)~~ Let $\sigma$ denote the restriction of $\lambda$ to $\ker(\lambda\mu)$. The listed relation is the only possible lift in its degree of the relation $\eta_\sigma \cdot \tau^2 h_\sigma = \rho_\sigma\alpha^2$ seen in \cref{lem:c2hidden}.

(6)~~ This holds as $\pi_{7-\mu-\lambda\mu}KO_A = 0$.
\end{proof}

This completes our computation of the ring structure of $\pi_\star KO_A$.

\subsection{Transfers}\label{ssec:kotransfer}

It remains only to understand the transfer. Fix a codimension $1$ subgroup $j\colon \ker(\lambda)\subset A$, and consider $j_!\colon \pi_{j^\ast\star}KO_{\ker(\lambda)}\rightarrow\pi_\star KO_A$.

\begin{lemma}\label{lem:transferko1}
$j_!\colon \pi_0 KO_{\ker(\lambda)}\rightarrow\pi_{1-\lambda}KO_A$ satisfies $j_!(1) = \rho_\lambda\alpha$.
\end{lemma}
\begin{proof}
The class $\rho_\lambda\alpha$ is in the kernel of $\rho_\lambda$, and thus in the image of $j_!$. This is the only possibility.
\end{proof}

\begin{lemma}\label{lem:kolambdanomial}
Fix a nontrivial functional $\lambda\in A^\vee$. Then any generator $x$ of the first form described in \cref{sssec:kobasis} may be written as
\[
x = \rho_{\lambda_1}\cdots\rho_{\lambda_n}\cdot\eta_{\mu_1}\cdots\eta_{\mu_s}\cdot t\cdot\beta^2 k_{H_1}\cdots \beta^2 k_{H_m}\cdot \tau^2_{\kappa_1}k_{H_{m+1}}\cdots\tau^2_{\kappa_t}k_{H_{m+t}},
\]
satisfying one of the following conditions:
\begin{enumerate}
\item $\lambda\notin\langle\lambda_1,\ldots,\lambda_n,\mu_1\ldots,\mu_s\rangle+H_1+\cdots+H_{m+t}$;
\item $\lambda\in\{\lambda_1,\lambda_1\lambda_2,\mu_1,\mu_1\mu_2,\lambda_1\mu_1\}$;
\item $\lambda\in H_1$;
\item $\lambda\in H_{m+1}$ and $\lambda = \kappa_1$;
\item $\lambda\in H_{m+1}$ and $\lambda\neq\kappa_1$.
\end{enumerate}
\end{lemma}
\begin{proof}
This follows immediately from \cref{lem:lambdanomial}.
\end{proof}

\begin{prop}\label{prop:kotransfer}
$j_!\colon \pi_{j^\ast\star} KO_{\ker(\lambda)}\rightarrow \pi_\star KO_A$ is determined by $\pi_\star KO_A$-linearity, comparison with $KU_A$, and \cref{lem:transferko1}.
\end{prop}

\begin{proof}
The proof is essentially identical to that of \cref{lem:transfer}. Fix $\xi\in RO(A)$, so that we must compute $j_!\colon \pi_{j^\ast\xi} KO_{\ker(\lambda)}\rightarrow\pi_\xi KO_A$. If $\pi_\xi KO_A$ is torsion-free, then $j_!$ is determined by comparison with $KU_A$. Thus we may suppose that $\pi_\xi KO_A$ is generated by a class of the form $x\alpha^\epsilon$, where $\epsilon\in\{1,2\}$ and $x$ is one of the types given in \cref{lem:kolambdanomial}. By $\pi_\star KO_A$-linearity, we further reduce to considering only the subwords which interact with $\lambda$.

We summarize the case analysis in the following table. The first column gives the form of the generators $x$ which one may reduce to considering, and the second column is a class $y$ such that $j^\ast(y)$ generates $\pi_{j^\ast\xi}KO_{\ker(\lambda)}$. In this case $j_!$ is determined by $j_!(j^\ast(y)) = j_!(1)\cdot y$; the third column gives $j_!(1)$ and the fourth column gives the product. When a particular $\epsilon$ is chosen, the claim is that with the other one would have $\pi_{j^\ast\xi}KO_A = 0$.

\begin{longtable}{l|l|l|l}
\toprule
$x$ & $y$ & $j_!(1)$ & $j_!(1)\cdot y$ \\
\midrule \endhead
\bottomrule \endfoot
$\alpha^\epsilon$ & $\alpha^\epsilon$ & $h_\lambda$ & $\rho_\lambda\eta_\lambda \alpha^\epsilon$ \\

$\rho_\lambda\alpha^\epsilon$ & $\alpha^{\epsilon-1}$ & $\rho_\lambda\alpha$ & $x$ \\

$\rho_{\lambda_1}\rho_{\lambda\lambda_1}\alpha^2$ & $\tau_{\lambda_1}^2 h_{\lambda_1}$ & $\rho_{\lambda\lambda_1}\eta_{\lambda_1}$ & $x$ \\

$\eta_\lambda\alpha$ & $\alpha^2$ & $0$ & $0$ \\

$\eta_{\lambda_1}\eta_{\lambda\lambda_1}\alpha^2$ & $\tau_{\lambda_1}^{-2}\beta^2 h_{\lambda_1}$ & $\rho_{\lambda_1}\eta_{\lambda\lambda_1}$ & $x$ \\

$\rho_{\lambda_1}\eta_{\lambda\lambda_1}\alpha^\epsilon$ & $\alpha^\epsilon$ & $\rho_{\lambda_1}\eta_{\lambda\lambda_1}$ & $x$ \\

$\beta^2 k_{\langle\lambda,\kappa\rangle}\alpha^2$ & $\rho_\kappa^2 \beta^4$ & $\tau_\kappa^{-2}k_{\langle\lambda,\kappa\rangle}$ & $x$ \\

$\tau^2_\lambda k_{\langle\lambda,\kappa\rangle}\alpha^2$ & $\tau_\kappa^4 \eta_\kappa^2$ & $\tau_\lambda^2\tau_\kappa^{-2} k_{\langle\lambda,\kappa\rangle}$ & $x$ \\

$\tau^2_\kappa k_{\langle\lambda,\kappa\rangle}\alpha^2$ & $\tau_\kappa^4\alpha^2$ & $\tau_\kappa^{-2}k_{\langle\lambda,\kappa\rangle}$ & $x$
\end{longtable}
\end{proof}

This concludes our computation of $\pi_\star KO_A$.

\begingroup
\raggedright 
\bibliographystyle{alpha}
\bibliography{refs} 
\endgroup

\end{document}